\documentclass[11pt]{article}
\usepackage{amsmath,amsfonts,amssymb,amsthm}
\usepackage{eucal}
\usepackage{color}
\usepackage{makeidx}
\usepackage{multicol}
\usepackage{verbatim}

\def\Z{{\mathbb Z}}

\def\R{{\mathbb R}}
\def\P{{\mathbb P}}
\def\E{{\mathbb E}}
\def\I{{\mathbb I}}
\newtheorem{theorem}{Theorem}
\newtheorem{lemma}[theorem]{Lemma} 
\newtheorem{corollary}[theorem]{Corollary}
\newtheorem{proposition}[theorem]{Proposition} 
\newtheorem{example}[theorem]{Example}

\makeindex

\title{Renewal Theory for Transient Markov Chains\\
with Asymptotically Zero Drift}

\author{Denis Denisov\footnote{University of Manchester, UK, 
denis.denisov@manchester.ac.uk},
Dmitry Korshunov\footnote{Lancaster University, UK, 
d.korshunov@lancaster.ac.uk},
and Vitali Wachtel\footnote{University of Augsburg, Germany, 
vitali.wachtel@math.uni-augsburg.de}}

\begin{document}
\maketitle

\begin{abstract}
We solve the problem of asymptotic behaviour of the renewal measure
(Green function) generated by a transient Lamperti's Markov chain $X_n$ in $\R$,
that is, when the drift of the chain tends to zero at infinity.
Under this setting, the average time spent by $X_n$ in the interval $(x,x+1]$
is roughly speaking the reciprocal of the drift
and tends to infinity as $x$ grows.

For the first time we present a general approach relying in a diffusion approximation to prove renewal theorems for Markov chains. 
We apply a martingale type technique and show that the asymptotic behaviour
of the renewal measure heavily depends on the rate at which the drift vanishes.
The two main cases are distinguished, either the drift of the chain decreases as
$1/x$ or much slower than that, say as $1/x^\alpha$ for some $\alpha\in(0,1)$.

The intuition behind how the renewal measure behaves
in these two cases is totally different.
While in the first case $X_n^2/n$ converges weakly to a $\Gamma$-distribution
and there is no law of large numbers available,
in the second case a strong law of large numbers holds true
for $X_n^{1+\alpha}/n$ and further normal approximation is available.

\vspace{5mm}


\end{abstract}

\section{Introduction}

Let $X=\{X_n, n\ge0\}$ be a time homogeneous 
Markov chain\index{Markov chain} whose
state space is some Borel subset $S$ of $\R$, 
that is, for all $x\in S$ and Borel sets $B\subseteq S$,
\begin{eqnarray*}
\P\{X_{n+1}\in B\mid X_0,\ldots,X_{n-1},X_n=x\}
&=& \P\{X_{n+1}\in B\mid X_n=x\}\\
&=:& P(x,B). 
\end{eqnarray*}
Standard examples of $S$ are $\R$, $\Z$, $\R^+$, and $\Z^+$.
In the sequel we just say that $X_n$
takes values in $\R$ keeping in mind that the corresponding transition 
probabilities may be defined on some subset $S$ of the real line only.

Denote by $\xi(x)$, $x\in\R$, a random variable corresponding
to the jump of the chain at point $x$, that is, 
a random variable with distribution
\begin{eqnarray*}
\P\{\xi(x)\in B\}
&=& \P\{X_{n+1}-X_n\in B\mid X_n=x\}\\
&=& \P_x\{X_{1}\in x+B\},
\quad B\in{\mathcal B}(\R);
\end{eqnarray*}
hereinafter the subscript $x$ denotes the initial position
of the Markov chain $X_n$, that is, $X_0=x$.
Denote the $k$th moment of the jump at point $x$ by 
\begin{eqnarray*}
m_k(x) &:=& \E\xi^k(x).
\end{eqnarray*}

Define the renewal (or potential)
kernel $Q$ by the equality
\begin{eqnarray*}
Q(\cdot,\cdot) &:=& \sum_{n=0}^\infty P^n(\cdot,\cdot).
\end{eqnarray*}
A Markov chain $X_n$ is called {\it transient} 
(see Meyn and Tweedie [\ref{Tw}, Ch. 8])
if there exists a countable cover of $\R$ by uniformly transient sets $\{B_k\}$.
In its turn a set $B\in{\mathcal B}(\R)$ is called {\it uniformly transient} if
\begin{eqnarray*}
\sup_{y\in B}Q(y,B) &<& \infty.
\end{eqnarray*}
By the Markov property, this is equivalent to the condition
\begin{eqnarray*}
\sup_{y\in\R}Q(y,B) &<& \infty,
\end{eqnarray*}
because, considering the first hitting time of $B$,
we observe by the Markov property that $Q(x,B)\le \sup_{y\in B}Q(y,B)$
for all states $x\in\R$.
If $X_n$ is transient with respect to the collection of intervals $B_k=(k,k+1]$, 
$k\in\Z$, then $Q(x,B)<\infty$ for all $x$ and bounded sets $B$ and,
hence, the renewal measure (Green function) generated by the chain $X_n$
\begin{eqnarray*}
H(B) &:=& \sum_{n=0}^\infty \P\{X_n\in B\},\quad B\in\mathcal B(\R),
\end{eqnarray*}
is finite for every initial distribution of the chain and bounded set $B$.

The main aim of the present paper is to study integral (elementary) 
and local renewal theorems for the renewal measure  $H$, 
that is we find asymptotics for $H(x_*,x]$, 
$H(x,x+t(x)]$, $H(x,x+h]$ as $x\to \infty$, where $t(x)$ is 
a growing function and $x_*$ and $h$ are some fixed constants. 

The simplest case of a transient Markov chain is just a random walk 
$X_n=\xi_1+\cdots+\xi_n$ generated by independent identically 
distributed random variables $\xi_n$'s with positive drift, 
which may be equivalently defined as a spatially and temporarily 
homogeneous Markov chain. 
The renewal theory for a random walk has been intensively studied since 1940s. 
The integral (elementary) renewal theorem for a random walk with positive
jumps and finite mean goes back to Feller \cite{Feller1941}
and states that $H(0,x]\sim x/\E\xi_1$ as $x\to\infty$. 
A more detailed information is available via the local renewal theorem, 
which was proved for lattice random variables in~\cite{EFP1949} and 
for non-lattice random variables in~\cite{B48}.
In the finite mean non-lattice case the local renewal theorem 
gives the following sharp asymptotics  
$H(x,x+h]\to h/\E\xi_1$ as $x\to\infty$, for any fixed $h>0$. 
Later Blackwell extended in~\cite{B53} the local renewal theorem 
to the case of i.i.d. random variables with positive mean 
that can take values of both signs using the important concept
of what was called by Feller ladder heights and ladder epochs. 
Original Blackwell's proof was considered to be quite complicated 
and a number of attempts were made to give an easier proof. 
A rather simple proof was given by Feller and Orey~\cite{FO61}, 
see also~\cite{Feller}. 
Further studies also considered behaviour of the remainder 
in the local renewal theorem, see~\cite{R1977}  and references therein. 
In the infinite mean case the asymptotics in Blackwell's theorem was not sharp. 
In 1960-70s a local renewal theorem was proved for regularly varying 
increments of index~$\alpha>1/2$, see~\cite{Garcia1962} and~\cite{E1970}. 
Subsequently there have been various improvements on these results, 
but the complete answer has been obtained very recently, 
see~\cite{CaravennaDoney2016}.

There exists a number of generalisations of the renewal theorem
for various stochastic processes. 
A natural extension is one for non-homogeneous (in time) random walk, 
that is a random walk with independent, 
but not necessarily identically distributed increments. 
Probably the first result in this direction was~\cite{CoxSmith1953}, 
where the local renewal theorem was derived from 
the local central limit theorem for  a non-homogeneous random walk. 
Further extensions may be found in~\cite{Smith1961, Williamson1965, Maejima1975}.
Renewal theorems for multidimensional random walks may be found 
in~\cite{Doney1966},\cite{Nagaev1980},~\cite{GH2013} and  recent paper~\cite{Berger2019}, see also references therein. 

The Markov setting has mostly been considered in the literature 
for case of Markov modulated random walks, 
see, e.g.~\cite{Kesten1974,AMN1978,KlP2003} and~\cite{Shurenkov1985}.
In this setting one can usually use the Harris regeneration 
and split the process into independent cycles. 
Then, the traditional setting of Blackwell's theorem can be used. 

For the results cited above, it is essential that
the underlying process possesses some independence structure. 
In the present paper we consider transient Markov chains 
where the cycle structure is not available, 
which makes a reduction to Blackwell's theorem impossible. 
Clearly, in order to observe some regular asymptotics 
for the renewal process, we need to assume some regular behaviour
of the Markov chain at infinity. In particular, if the drift of $X_n$,
$m_1(x)$, has a positive limit at infinity, say $a$, then
the local renewal result, $H(x,x+h]\to\ h/a$, 
is only known for an {\it asymptotically homogeneous in space} Markov chain
which is defined as a Markov chain such that, for some random variable $\xi$,
\begin{equation}\label{asymp.hom}
\xi(x) \Rightarrow \xi\quad\mbox{as }x\to\infty,
\end{equation}
see~\cite{Korshunov2008}; 
if there is no asymptotic homogeneity in space
then the asymptotic behaviour of $H(x,x+h]$ may be very different. 

So, while the asymptotic behaviour of a Markov chain with asymptotically
non-zero drift is well understood, at least if it is asymptotically
homogeneous, the case of a  drift vanishing at infinity
is studied much less. In general, we say that a Markov chain has 
{\it asymptotically (in space) zero drift} if $m_1(x)=\E\xi(x)\to 0$ as $x\to\infty$.
The study of processes with asymptotically zero drift was initiated by Lamperti
in a series of papers~\cite{Lamp60,Lamp62,Lamp63}. 
The vanishing drift seems to be more difficult for the analysis due to 
the fact that the Markov chain tends to infinity much slower 
and one should take into account diffusion  fluctuations.  

An integral (elementary) renewal theorem for a transient Markov chain with 
drift $m_1(x)$ asymptotically proportional to $1/x$ at infinity
was proved in~\cite{DKW2013}; it was shown there that then 
the renewal function behaves as $cx^2$ for large values of $x$. 

Here we present for the first time a unified  approach that allows 
us to prove renewal theorems for general Markov chains.
This is the  {\it main novelty} of the present paper. 
In this paper we analyse one-dimensional Markov chains, but clearly the approach 
suggested below can be used in the multidimensional setting as well. 
Our approach relies on the diffusion approximation, for that reason 
we consider Markov chains which may be approximated by diffusion process. 
Then, if we have some result of renewal type for a diffusion processes
we should be able to obtain a similar result for a Markov chain
having similar asymptotic behaviour of the first two moments of jumps.  
In particular, we will see in the examples below that as soon as we have 
a Green function for the diffusion process we should, in principle, 
be able to construct an approximation for the Green function of the 
Markov chain  and thus to derive a renewal theorem.

\subsection{Main results on renewal measure}

Throughout we assume some weak irreducibility of $X_n$,
namely that there are no bounded trajectories of $X_n$,
that is, 
\begin{equation}\label{eq:irreducibility}
\limsup_{n\to\infty} X_n = \infty\ \mbox{ a.s. }
\end{equation}
For any $s>0$ we denote the $s$-truncation 
of the $k$th moment of jump at state $x$ by
\begin{eqnarray*}
m_k^{[s]}(x) &:=& \E\{\xi^k(x);\ |\xi(x)|\le s\}.
\end{eqnarray*}
For any random variables $\xi$ and $\eta$ we write 
$\xi \le_{\rm st}\eta $ if $\P\{\xi>t\}\le \P\{\eta>t\}$ for all $t\in\R$.

\begin{theorem}\label{thm:regular}
Let $X_n$ be such that~\eqref{eq:irreducibility} holds and 
\begin{eqnarray}\label{m1.m2.1x}
m_1^{[s(x)]}(x)\sim\frac{\mu}{x},&&\quad m_2^{[s(x)]}(x)\to b \in(0,\infty)
\quad\mbox{as }x\to\infty, 
\end{eqnarray}
for some $\mu>b/2$ and an increasing level $s(x)$
of order $o(x)$. Assume also that, 
\begin{eqnarray}\label{regular_left_tail}
\P\{|\xi(y)|\ge s(y)\} &\le& p(y)/y,  
\end{eqnarray}
for some decreasing integrable at infinity function $p(x)$, and
\begin{eqnarray}\label{majorant_third}
|\xi(y)| \I\{|\xi(y)|\le s(y)\} &\le_{\rm st}& 
\widehat\xi\quad\mbox{for all }y\ge 0,
\end{eqnarray}
where 
\begin{equation}\label{majorant_third_moment_exists}
\E\widehat\xi^2<\infty. 
\end{equation}	
Then, for every function $h(x)\uparrow\infty$ of order $o(x),$ 
we have
$$
H(x,x+h(x)]\sim\frac{2}{2\mu-b}xh(x)\quad\mbox{as }x\to\infty.
$$
\end{theorem}

Notice that the both conditions~\eqref{regular_left_tail} 
and~\eqref{majorant_third} are met if $|\xi(y)|\le_{\rm st}\xi$
for all $y$ and for some $\xi$ satisfying~\eqref{majorant_third_moment_exists}.

In the course of the proof of this and subsequent theorems we construct 
a bounded non-negative supermartingale, which also shows that $X_n\to\infty$ a.s. 
This convergence clearly implies that $X_n$ is transient.  
Transience in the case of $\mu>b/2$ was considered under various additional 
conditions in a series of papers, 
see e.g.~\cite[Theorem 3.1]{Lamp60} or~\cite[Theorem 3.2.3]{MPW16}.

Under slightly stronger assumptions,
an integral renewal theorem was proved in~\cite[Theorem 5]{DKW2013}
where it was shown that $H(0,x]\sim x^2/(2\mu-b)$ as $x\to\infty$.

We now turn to the critical case $\mu=b/2$ where the properties 
of the chain---particularly recurrence and transience---depend 
on further terms in  asymptotic expansions for the moments of increments. 
As the next theorem shows this is also true for the renewal function.

\begin{theorem}\label{thm:critical}
Let $X_n$ be such that~\eqref{eq:irreducibility} holds
and that there exist $m\ge1$, $\gamma>0$ and an increasing level
$s(x)$ of order $o(x)$ such that
\[
\frac{2m_1^{[s(x)]}(x)}{m_2^{[s(x)]}(x)}=\frac{1}{x}+\frac{1}{x\log x}
+\ldots+\frac{1}{x\log x\ldots\log_{(m-1)}x}+\frac{\gamma+1+o(1)}{x\log x\ldots\log_{(m)}x}
\]
and $m_2^{[s(x)]}(x)\to b>0$ as $x\to\infty$. 
Assume that, for some $\varepsilon>0$,
\begin{eqnarray}\label{critical_left_tail}
\P\{|\xi(x)|>s(x)\} &=& o(1/x^2\log^{2+\varepsilon}x),\\
\label{critical_m3}
\E\{|\xi(x)|^3;\ |\xi(x)|\le s(x)\} &=& o(x/\log^{1+\varepsilon}x),\\ 
\label{majorant_third_critical}
|\xi(y)| \I\{|\xi(y)|\le s(y)\} &\le_{\rm st}& \widehat\xi,
\end{eqnarray}
where $\widehat\xi$ satisfies \eqref{majorant_third_moment_exists}.
Then, for every function $h(x)\uparrow\infty$ of order $o(x)$, we have
\begin{eqnarray*}
H(x,x+h(x)] &\sim& \frac{2h(x)}{b\gamma}x\log x\ldots\log_{(m)}x
\quad\mbox{as }x\to\infty.
\end{eqnarray*}
\end{theorem}

Transience in a similar setting goes back to~\cite[Theorem 3]{AMI}.

As we have mentioned above, the integral renewal theorem 
in the case $\mu>b/2$ was proved in \cite{DKW2013}. 
The proof in that paper is based on the convergence of $X_n^2/n$ towards 
$\Gamma$-distribution. This approach is not applicable under the conditions
of Theorem~\ref{thm:critical}, although the convergence to 
$\Gamma$-distribution is still valid. 
The reason is that some chains with $\mu=b/2$ are null-recurrent 
while other are transient, but this difference disappears in the weak limit. 
The only statement which can be obtained from weak convergence here
is the following lower bound:
\[
\lim_{x\to\infty}\frac{H(0,x]}{x^2}=\infty
\]

In the next theorem we consider the case where the drift 
decreases slower than $1/x$, that is, $m_1(x)x\to\infty$.

\begin{theorem}\label{thm:weibull}
Let $X_n$ satisfy the condition \eqref{eq:irreducibility} and 
$v$ be a decreasing function such that $xv(x)\to\infty$ and $v'(x)=o(v^2(x))$.
Let there exist an increasing level $s(x)$ of order $o(1/v(x))$ such that 
\[
m_1^{[s(x)]}(x)\sim v(x),\quad m_2^{[s(x)]}(x)\to b \in(0,\infty)
\quad\mbox{as }x\to\infty, 
\]
where $v$ is a decreasing function such that $xv(x)\to\infty$ 
and $v'(x)=o(v^2(x))$. Assume also that,
\begin{align}\label{regular_left_tail_weibull}
&\P\{|\xi(y)|\ge s(y)\}\le p(y)v(y),\\
\label{majorant_third_weibull}
& |\xi(y)| \I\{|\xi(y)|<s(y) \} \le_{\rm st} \widehat\xi
\quad \mbox{for all }y\ge0,  
\end{align}
where $p(x)$ is a non-increasing, non-negative integrable at infinity function,	
and $\widehat\xi$ satisfies \eqref{majorant_third_moment_exists}.
Then, for every function $h(x)\uparrow\infty$ of order $o(1/v(x))$, we have
$$
H(x,x+h(x)]\sim\frac{h(x)}{v(x)}\quad\mbox{as }x\to\infty.
$$
\end{theorem}

In the two examples---nearest neighbour Markov chain and diffusion 
process---considered in the two subsections below
it is possible to construct an appropriate martingale 
which allows us to find the renewal measure in a closed form. 
For general Markov chains considered in the last three theorems, 
this martingale approach does not work because it is hopeless to construct 
such a martingale. 
However, it is possible to construct almost 
a martingale that allows us to derive the asymptotic behaviour
of the renewal measure; it is done in Section~\ref{sec:ilrtgi}.

While the asymptotic behaviour of the renewal measure on growing intervals
is derived under assumptions on regular behaviour of the first two
moments only, it seems that the local renewal theorem can be only proved
for asymptotically homogeneous in space Markov chain.
The next result gives us a tool for deriving asymptotic behaviour
of the renewal measure on intervals from results for sufficiently slowly
growing intervals. It requires weak convergence of jumps, 
see~\eqref{asymp.hom}.

\begin{theorem}\label{thm:ah.renewal}
Let~\eqref{asymp.hom} hold and the family of random variables 
$\{|\xi(x)|,\ x\in\R\}$ admit an integrable majorant $\Xi$, 
that is, $\E\Xi<\infty$ and
\begin{eqnarray}\label{majoriz}
|\xi(x)| &\le_{\rm st}& \Xi
\quad\mbox{for all }x\in\R.
\end{eqnarray}
Assume that there exist a bounded function $v(x)>0$, 
a growing level $\widetilde t(x)\uparrow\infty$ 
and a constant $C_H<\infty$ such that, 
for any $t(x)\uparrow\infty$ satisfying $t(x)\le\widetilde t(x)$,
\begin{equation}\label{eq.growing.intervals}
\frac{v(x)H(x,x+t(x)]}{t(x)}\ \to\ C_H\quad\mbox{as }x\to\infty.
\end{equation}

If the limiting random variable $\xi$ is non-lattice, 
then $v(x)H(x,x+h]\to C_Hh$ as $x\to\infty$, for all fixed $h>0$.

If the chain $X_n$ is integer valued and $\Z$ is the minimal lattice 
for the variable $\xi$, then $v(k)H\{k\}\to C_H$ as $k\to\infty$.
\end{theorem}

Let us apply the last result to chains considered in 
Theorems~\ref{thm:regular}--\ref{thm:weibull}.

\begin{corollary}\label{cor:regular}
Under the conditions of Theorem \ref{thm:regular},
\eqref{asymp.hom} and \eqref{majoriz}, we have, for every $h>0$, 
\begin{eqnarray*}
H(x,x+h] &\sim& \frac{2h}{2\mu-b}x\quad\mbox{as }x\to\infty,
\end{eqnarray*}
if the limiting random variable $\xi$ is non-lattice, and
\begin{eqnarray*}
H\{k\} &\sim& \frac{2}{2\mu-b}k\quad\mbox{as }k\to\infty,
\end{eqnarray*}
if the chain $X_n$ is integer valued and $\Z$ is the minimal lattice 
for the variable $\xi$.
\end{corollary}

\begin{corollary}\label{cor:critical}
Under the conditions of Theorem \ref{thm:critical},
\eqref{asymp.hom} and \eqref{majoriz}, we have, for every $h>0$,
\begin{eqnarray*}
H(x,x+h] &\sim& \frac{2h}{b\gamma}x\log x\ldots\log_{(m)}x
\quad\mbox{as }x\to\infty,
\end{eqnarray*}
if the limiting random variable $\xi$ is non-lattice, and
\begin{eqnarray*}
H\{k\} &\sim& \frac{2}{b\gamma}k\log k\ldots\log_{(m)}k\quad\mbox{as }k\to\infty,
\end{eqnarray*}
if the chain $X_n$ is integer valued and $\Z$ is the minimal lattice 
for the variable $\xi$.
\end{corollary}

\begin{corollary}\label{cor:weibull}
Under the conditions of Theorem~\ref{thm:weibull},
\eqref{asymp.hom} and \eqref{majoriz}, we have, for every $h>0$,
\begin{eqnarray*}
H(x,x+h] &\sim& \frac{h}{v(x)}\quad\mbox{as }x\to\infty,
\end{eqnarray*}
if the limiting random variable $\xi$ is non-lattice, and
\begin{eqnarray*}
H\{k\} &\sim& \frac{1}{v(k)}\quad\mbox{as }k\to\infty,
\end{eqnarray*}
if the chain $X_n$ is integer valued and $\Z$ is the minimal lattice 
for the variable $\xi$.
\end{corollary}

Markov chains with asymptotically zero drift naturally appear in various areas
including branching processes, stochastic difference equations, networks, etc.
In most cases we are aware of it gives rise to the drift of order $O(1/x)$ 
at infinity. We now consider the random walk conditioned 
to stay positive, which represents one of the classical examples of chains with asymptotically zero drift.

Let $S_n$ be a random walk with independent identically distributed
increments $\xi_k$, that is, $S_n=\xi_1+\xi_2+\ldots+\xi_n$, $n\ge 1$.
Let $\tau_x$ be the first time when $S_n$ started at $x$ is non-positive:
$$
\tau_x:=\min\{n\ge1:x+S_n\le 0\}.
$$
We assume that the random walk $S_n$ is oscillating, that is,
$$
\liminf_{n\to\infty} X_n=-\infty,\quad
\limsup_{n\to\infty} X_n=\infty\quad\mbox{with probability 1.}
$$
In particular, $\P\{\tau_x<\infty\}=1$ for all starting points $x$.
Let $\chi^-$ denote the first weak descending ladder height of $S_n$,
that is, $\chi^-=-S_{\tau_0}$. Let $V(x)$ denote the renewal function 
corresponding to weak descending ladder heights of our random walk:
$$
V(x):=1+\sum_{k=1}^\infty\P\{\chi^-_1+\chi^-_2+\ldots+\chi^-_k<x\},
$$
where $\chi^-_k$ are independent copies of $\chi^-$.

It is well-known---see e.g. Bertoin and Doney \cite{BD94}---that
$V(x)$ is a harmonic function for $S_n$ killed at leaving
$(0,\infty)$. More precisely,
$$
V(x)=\E\{V(x+S_1); \tau_x>1\},\quad x\ge 0.
$$
This implies that Doob's $h$-transform
\begin{equation}
\label{def.cond}
P(x,dy):=\frac{V(y)}{V(x)}\P\{x+S_1\in dy,\tau_x>1\}
\end{equation}
defines a stochastic transition kernel on $\R^+$.
Let $X_n$ be the corresponding Markov chain, which we shall call 
{\it random walk conditioned to stay positive.}

\begin{example}\label{thm:rwcsp}
Let $\E\xi_1=0$ and $\sigma^2:=\E\xi_1^2\in(0,\infty)$.
Then the renewal measure $H$ of the random walk conditioned to stay positive
has the following asymptotic behaviour: for every fixed $h>0$,
$$
H(x,x+h]\sim\frac{2h}{\sigma^2}x\quad\mbox{as }x\to\infty
$$
if $\xi_1$ is non-lattice, and
$$
H\{k\}\sim\frac{2}{\sigma^2}k\quad\mbox{as }k\to\infty,\ k\in\Z,
$$
if $\Z$ is the minimal lattice for $\xi_1$.
\end{example}

In Section \ref{sec:rwcsp},
we provide a proof based on Corollary \ref{cor:regular}.
It is worth mentioning that the finiteness of $\E\xi_1^2$ 
does not imply existence of second moments of $X_n$. 
Thus, this example underlines the advantage of our assumptions in terms 
of truncated moments. 
Let us note that one can also prove the last result making use of Proposition 19.3 
from Spitzer \cite{Spitzer} in lattice case and Port and Stone \cite{PortStone} 
in non-lattice case, together with the well-known result on the renewal 
measures of the descending and ascending ladder height processes associated 
with the random walk.

For Markov chains considered in Theorems~\ref{thm:regular} 
and \ref{thm:critical} one knows that $X_n^2/n$ converges in distribution towards 
a $\Gamma$-distribution. Since this convergence takes place without centering, 
$X_n$ tends to infinity diffusively slow. The influence of the diffusion component is expressed by the fact that the renewal function grows at rate $2x/(2\mu-b)$
which is strictly greater than the reciprocal of the drift at point $x$.
The chains satisfying the conditions of Theorem~\ref{thm:weibull} go to 
infinity much faster, and a law of large numbers holds. More precisely, 
if the drift is of order $x^{-\alpha}$, $\alpha\in(0,1)$ then $X_n^{1+\alpha}/n$ 
converges to a positive constant. As a result, we have the classical form of 
the local renewal theorem: the rate of growth is asymptotically equivalent 
to the reciprocal of the drift. We believe that Theorem~\ref{thm:weibull} 
remains valid for chains with unbounded second moments, 
but we do not know how to prove it.

We conclude this section by noting that Markov chains with growing 
second moments can be transformed sometimes to chains with bounded second moments.
First we consider a critical branching process with immigration. 
Let $\{\zeta_{n,k}\}_{n,k\ge1}$ be independent copies of a random variable 
$\zeta$ valued in $\mathbb{Z}^+$. Assume that $\E\zeta=1$ and 
$\sigma^2:=\E\zeta^2\in(0,\infty)$. Let $\{\eta_n\}_{n\ge1}$
be i.i.d. random variables which are also independent of
$\{\zeta_{n,k}\}$. Assume that $a:=\E\eta_1>0$ and $\E\eta_1^2<\infty$.
Consider the Markov chain
\begin{equation*}
Z_{n+1}=\sum_{k=1}^{Z_n}\zeta_{n+1,k}+\eta_{n+1},\quad
n\ge0.
\end{equation*}
For this chain one has
\begin{eqnarray*}
\E\{Z_1-Z_0|Z_0=k\}&=&a,\\
\E\{(Z_1-Z_0)^2|Z_0=k\}&=&\sigma^2 k+\E\eta_1^2.
\end{eqnarray*}
Since the second moments of increments are linearly growing we cannot apply 
our results directly to this chain. 
However one can consider the chain $X_n=\sqrt{Z_n}$,
which then satisfies the conditions of Theorem~\ref{thm:regular} with 
$\mu=(a-\sigma^2/4)/2$ and $b=\sigma^2/4$. 
Furthermore, the central limit theorem implies that $X_n$ is asymptotically 
homogeneous in space and the limiting variable $\xi$ is normally distributed 
with parameters $0$ and $\sigma^2/4$. Then, applying Corollary~\ref{cor:regular}, 
we obtain the local renewal theorem for $X_n$. Performing the inverse 
transformation, one gets asymptotics for the renewal function of $Z_n$ 
on the intervals $[k,k+h\sqrt{k})$. Unfortunately, our approach does not 
allow us to consider shorter intervals. An integral renewal theorem for 
$Z_n$ has been obtained by Pakes~\cite{Pakes1972}, while Mellein~\cite{Mellein1983} 
has proved the corresponding local renewal theorem. 
Their proofs use the machinery of generating functions.

In general, if the second moments of jumps of $X_n$ are growing 
as $x^\beta$, $\beta\in(0,2)$, then the jumps of $X_n^{1-\beta/2}$ 
have bounded second moments and one can try to apply one of our theorems.

\subsection{Key renewal theorem}
We now turn to the renewal equation
\begin{equation*}
Z(B)\ =\ z(B)+\int_\R Z(dy)\, P(y,B),\ \ B\in\mathcal B(\R),
\end{equation*}
where $z$ is a finite non-negative measure on $\R$. 
This is more than sufficient to ensure that
$$
Z(B)=\int_\R z(du)H_u(B),\ \ B\in\mathcal B(\R),
$$
is a unique locally finite solution to the renewal equation.  
The analysis of the preceding subsection of this paper
allows us to deduce the asymptotic behaviour of the measure $Z$ at infinity.
The proof is immediate from the dominated convergence theorem.

\begin{theorem}\label{reneqn}
Let $B\in\mathcal B(\R)$. 
Assume that, for some positive function $g(x)$ and for all $y\in\R$, 
$$
H_y(x+B)\ \sim\ g(x)\quad\mbox{as }x\to\infty,
$$
and, for some $c<\infty$,
$$
H_y(x+B)\ \le\ cg(x)\quad\mbox{for all }x,\ y\in\R.
$$
If $z$ is a finite measure, then
$$
Z(x+B)\ \sim\ z(\R)g(x)\quad\mbox{as }x\to\infty.
$$
\end{theorem}

\subsection{Nearest-neighbour Markov chain}

To illustrate the approach and intuition beyond the results above,
we consider first a nearest-neighbour
(skip-free or continuous) Markov chain $X_n$ on $\Z^+$, 
that is, $\xi(x)$ only takes values $-1$, $1$ and $0$,
with probabilities $p_-(x)$, $p_+(x)$ and $p_0(x)=1-p_-(x)-p_+(x)$
respectively, $p_-(0)=0$. 
Nearest-neighbour Markov chains are very useful for our purposes
because in this case one can write down an expression for the renewal measure 
in a closed form and then analyse it. 

For a nearest-neighbour Markov chain $X_n$ with specific jump probabilities,
$p_-(x)=(1-\lambda/(x+\lambda))/2$ and $p_+(x)=(1+\lambda/(x+\lambda))/2$, 
$\lambda>-1/2$ (which corresponds to transience of $X_n$),
Guivarc'h et al. \cite[Theorems 42 and 43]{GKR} 
have obtained weak convergence of $X_n^2/n$ to the
$\Gamma_{\lambda+1/2,2}$-distribution and the local renewal theorem in that case. 
They used the technique of orthogonal polynomials,  
in  Rosenkrantz~\cite{Rosenkrantz1966}, which is only available 
for specific Markov chains considered in that paper.

Let 
$$
p_+(k)=p+\varepsilon_+(k)\quad\mbox{and}\quad 
p_-(k)=p-\varepsilon_-(k),\quad p\le 1/2.
$$
Assume that $\varepsilon_\pm(k)\to 0$ as $k\to\infty$
that is the case of asymptotically zero drift 
and convergent second moment of jumps, $m_2(k)\to 2p$ as $k\to\infty$.

We  can define the renewal measure  (Green function) of $X_n$ as follows 
\begin{eqnarray*}
h_{x_0}(x) &:=& \sum_{n=0}^\infty \P_{x_0}\{X_n=x\},\quad x_0,\ x\in\Z^+.
\end{eqnarray*}
Since we consider a Markov chain with jumps $-1$, $1$ and $0$ only,
$h_{x_0}(x)=h_x(x)$ for all ${x_0}\le x$. 
Below we demonstrate how to find $h_{x_0}(x)$ in a closed form.
 
We first look for a function $g(x,z)\ge 0$ such that,
for all $x$, the process 
\begin{eqnarray}\label{nnmc.rW}
W_n &=& g(x,X_n)-\sum_{k=0}^{n-1}\I\{X_k=x\}
\end{eqnarray}
is a martingale which happens if $g$ satisfies 
the following system of equations
\begin{eqnarray*}
g(x,0) &=& p_0(0)g(x,0)+p_+(0)g(x,1),\\
g(x,y) &=& p_-(y)g(x,y-1)+p_0(y)g(x,y)+p_+(y)g(x,y+1) - \I\{y=x\},
\quad y\ge 1.
\end{eqnarray*}
Take $g(x,0)=g(x,1)=\ldots=g(x,x)=0$. Then for $y=x$ we get
\begin{eqnarray*}
g(x,x+1)-g(x,x) &=& g(x,x+1)\ =\ \frac{1}{p_+(x)}
\ =\ \frac{1}{p_-(x)}\frac{p_-(x)}{p_+(x)},
\end{eqnarray*}
and, for $y\ge x+1$,
\begin{eqnarray*}
g(x,y+1)-g(x,y) &=& \frac{p_-(y)}{p_+(y)}(g(x,y)-g(x,y-1))\\
&=& \prod_{z=x+1}^y\frac{p_-(z)}{p_+(z)}(g(x,x+1)-g(x,x))\\
&=& \frac{1}{p_+(x)}\prod_{z=x+1}^y\frac{p_-(z)}{p_+(z)}\\
&=& \frac{1}{p_-(x)}\prod_{z=x}^y\frac{p_-(z)}{p_+(z)}.
\end{eqnarray*}
Therefore, for $y\ge x+1$,
\begin{eqnarray*}
g(x,y)\ =\ \sum_{u=x}^{y-1} (g(x,u+1)-g(x,u))
&=& \frac{1}{p_+(x)}\sum_{u=x}^{y-1}
\prod_{z=x+1}^u\frac{p_-(z)}{p_+(z)}\\
&=& \frac{1}{p_-(x)}\sum_{u=x}^{y-1}
\prod_{z=x}^u\frac{p_-(z)}{p_+(z)},
\end{eqnarray*}
which is increasing in $y$. This sequence is bounded provided
\begin{eqnarray}\label{nnmc.rcond}
\sum_{u=1}^\infty \prod_{z=1}^u\frac{p_-(z)}{p_+(z)} &<& \infty.
\end{eqnarray}
Then denote
\begin{eqnarray*}
g(x,\infty)\ :=\ \lim_{y\to\infty} g(x,y) &=& \frac{1}{p_+(x)}
\sum_{u=x}^\infty \prod_{z=x+1}^u\frac{p_-(z)}{p_+(z)}.
\end{eqnarray*}
The sequence \eqref{nnmc.rW} is a martingale, so for all $n$,
\begin{eqnarray*}
g(x,x_0)\ =\ \E_{x_0} W_0\ =\ \E_{x_0} W_n &=& 
\E_{x_0} g(x,X_n)-\E_{x_0}\sum_{k=0}^{n-1}\I\{X_k=x\}
\end{eqnarray*}
and hence
\begin{eqnarray*}
\sum_{k=0}^{n-1}\P_{x_0}\{X_k=x\} &=& 
\E_{x_0} g(x,X_n)-g(x,x_0)\ <\ g(x,\infty)\ <\ \infty.
\end{eqnarray*}
Finiteness of the renewal measure implies the transience of $X_n$,
hence $X_n\to\infty$ as $n\to\infty$ a.s. 
Thus, we get the following explicit representation for the renewal measure
\begin{eqnarray}\label{nnmc.rA}
h_{x_0}(x) \ =\ g(x,\infty)-g(x,x_0) &=& \frac{1}{p_+(x)}
\sum_{u=x\vee x_0}^\infty \prod_{z=x+1}^u\frac{p_-(z)}{p_+(z)}\nonumber\\
&=& \frac{1}{p_-(x)}
\sum_{u=x\vee x_0}^\infty \prod_{z=x}^u\frac{p_-(z)}{p_+(z)}\nonumber\\
&=& \frac{1}{p_+(x)} \prod_{z=1}^x\frac{p_+(z)}{p_-(z)}
\sum_{u=x\vee x_0}^\infty \prod_{z=1}^u\frac{p_-(z)}{p_+(z)}.
\end{eqnarray}
We have
\begin{eqnarray*}
\prod_{z=x}^u\frac{p_-(z)}{p_+(z)} &=& 
\exp\biggl\{
\sum_{z=x}^u\log \frac{1-\varepsilon_-(z)/p}{1+\varepsilon_+(z)/p} 
\biggr\}. 
\end{eqnarray*}
Assume that
\begin{eqnarray}\label{assump.r}
\frac{2m_1(x)}{m_2(x)} &\sim& r(x)\quad\mbox{as }x\to\infty,
\end{eqnarray}
where $r(x)$ is a differentiable decreasing function such that
the quotient $r'(x)/r^2(x)$ has a limit at infinity.
The last asymptotic equivalence is equivalent to 
\begin{eqnarray*}
\log \frac{1-\varepsilon_-(x)/p}{1+\varepsilon_+(x)/p} &\sim& -r(x)
\quad\mbox{as }x\to\infty.
\end{eqnarray*}
Fix an $\varepsilon>0$. Then for all sufficiently large $x$ we can write 
\[
-(1+\varepsilon) r(x)\ \le\ \log\frac{1-\varepsilon_-(x)/p}{1+\varepsilon_+(x)/p}
\ \le\ -(1-\varepsilon) r(x). 
\] 
Therefore, for such $x$, we have the following upper bound 
\begin{eqnarray*}
h_{x_0}(x) &\le& \frac{1}{p_-(x)}\sum_{u=x}^\infty
\exp\biggl\{-(1-\varepsilon) \sum_{z=x}^u r(z) \biggr\}\\
&\le& \frac{1}{p_-(x)}\int_x^\infty
\exp\biggl\{-(1-\varepsilon) \int_x^u r(z) dz \biggr\}du,
\end{eqnarray*}    
due to the decrease of $r(z)$. Putting 
\[
U_\varepsilon(x)= \int_x^\infty
\exp\biggl\{-(1-\varepsilon) \int_0^u r(z) dz \biggr\}du 
\]
we observe that 
\[
\int_x^\infty \exp\biggl\{-(1-\varepsilon) \int_x^u r(z)dz\biggr\}du 
=\frac{U_\varepsilon(x)}{-U_\varepsilon'(x)}.
\]
By L'Hospital's rule,  
\begin{eqnarray*}
\lim_{x\to\infty}\frac{U_\varepsilon(x)}{-U_\varepsilon'(x)/r(x)} &=&  
\lim_{x\to\infty}\frac{U_\varepsilon'(x)}
{-U_\varepsilon''(x)/r(x)+U_\varepsilon'(x)r'(x)/r^2(x)}\\
&=& \frac{1}{1-\varepsilon+\lim_{x\to\infty}r'(x)/r^2(x)}.
\end{eqnarray*}    
Therefore, 
\begin{eqnarray*}
\limsup_{x\to\infty} h_{x_0}(x)r(x) &\le& 
\frac{1}{p}
\frac{1}{1-\varepsilon+\lim_{x\to\infty}r'(x)/r^2(x)}.
\end{eqnarray*}    
Similarly, starting from inequalities
\begin{eqnarray*}
h_{x_0}(x) &\ge& \frac{1}{p_+(x)}\sum_{u=x}^\infty
\exp\biggl\{-(1+\varepsilon) \sum_{z=x+1}^u r(z) \biggr\}\\
&\ge& \frac{1}{p_+(x)}\int_x^\infty
\exp\biggl\{-(1-\varepsilon) \int_x^u r(z) dz \biggr\}du,
\end{eqnarray*}
we get a lower bound  
\begin{eqnarray*}
\liminf_{x\to\infty} h_{x_0}(x)r(x) &\ge& 
\frac{1}{p}
\frac{1}{1+\varepsilon+\lim_{x\to\infty}r'(x)/r^2(x)}.
\end{eqnarray*}
Since $\varepsilon>0$ is arbitrary we conclude that 
\begin{eqnarray*}
h_{x_0}(x) &\sim& \frac{1}{pr(x)}
\frac{1}{1+\lim_{y\to\infty}r'(y)/r^2(y)}
\quad\mbox{as }x\to\infty.
\end{eqnarray*}

In the following two examples we consider canonical drifts
where $r'(y)/r^2(y)$ has either negative or zero limit at infinity.

\begin{example}\label{ex:nnmc.h.asy.1x}
If $\varepsilon_+(k)\sim \mu_+/k$ and
$\varepsilon_-(k)\sim \mu_-/k$ as $k\to\infty$ and $\mu:=\mu_++\mu_->p$,
then \eqref{assump.r} is valid with $r(x)=\mu/px$,
$r'(x)/r^2(x)\to -p/\mu$, and we deduce that  
\begin{eqnarray*}
h_{x_0}(x) &\sim& \frac{x}{\mu-p}\quad\mbox{as }x\to\infty.
\end{eqnarray*}
\end{example}

\begin{example}\label{ex:nnmc.h.asy.alpha}
If $\varepsilon_+(k)\sim \mu_+/k^\alpha$ and 
$\varepsilon_-(k)\sim \mu_-/k^\alpha$ as $k\to\infty$, 
$\mu:=\mu_++\mu_->0$, $\alpha\in(0,1)$, 
then \eqref{assump.r} is valid with $r(x)=\mu/px^\alpha$,
$r'(x)/r^2(x)\to 0$, and we deduce a Weibullian asymptotics for 
the renewal measure at infinity,
\begin{eqnarray*}
h_{x_0}(x) &\sim& \frac{x^\alpha}{\mu}
\ \sim\ \frac{1}{m_1(x)}\quad\mbox{as }x\to\infty.
\end{eqnarray*}
\end{example}

Let us note that a lower bound in Example \ref{ex:nnmc.h.asy.1x} may be deduced 
from the local limit theorem from Alexander \cite[Theorem 2.4]{Alex11}.

\subsection{Diffusion process}

Now let us consider another Markov process allowing solutions in closed form, 
a transient diffusion $X_t$ on $\R$ 
(or $\R^+$) with the following generator
\begin{align*}
A = \mu(x)\frac{d}{dx}+\frac{\sigma^2(x)}{2}\frac{d^2}{dx^2}.
\end{align*}
We consider a regular diffusion, in the sense of properties (i)-(iii) 
of \cite[Chapter VII.3]{RY1999}. 
For the transience it is sufficient to assume that the following function 
\begin{eqnarray}\label{def:U.dif}
U(x) &:=& \int_x^\infty 
\exp\biggl\{-\int_0^v \frac{2\mu(y)}{\sigma^2(y)} dy\biggr\} dv
\end{eqnarray}
is finite for all $x$. This function solves the homogeneous equation
\begin{eqnarray}\label{eq:U}
AU &=& 0.
\end{eqnarray}
In this case $X_t\to \infty $ a.s.\
and we are interested in the continuous time analogue of the renewal function,
\[
H_y(x,x+h]\ :=\ \int_0^\infty \P_y\{X_t\in(x,x+h]\}dt.
\] 

It is known that the process $f(X_t)-f(X_0)-\int_0^t Af(X_s)ds $ is a local martingale. 
Fix $x$ and $h$.
Suppose we can find a bounded function $f(z)=f_{h,x}(z)$ 
such that $f(z)\to 0$ as $z\to\infty$ and 
\begin{equation}\label{eq:ode.diffusion}
Af(z)= - \I\{z\in (x,x+h]\}.
\end{equation}
Then the optional stopping theorem and a.s.\ 
convergence $X_t\to\infty$ as $t\to\infty$ will give us an equality
\[
f(y)=\E_y f(X_0) = \E_y\biggl[\int_0^{\infty} \I\{X_t\in(x,x+h]\} dt\biggr]
=H_y(x,x+h],
\]
which allows us to analyse $H_y$. 

So, we need to solve the ordinary differential equation~\eqref{eq:ode.diffusion}. 
To this end, consider 
\[
m(x)\ :=\ \int_0^x \frac{2dv}{-U'(v)\sigma^2(v)}\ =\
\int_0^x \frac{2}{\sigma^2(v)}
\exp\biggl\{\int_0^v \frac{2\mu(y)}{\sigma^2(y)} dy\biggr\}dv
\]
and then
\[
G_x(z)\ :=\ 
\begin{cases}
U(z)m(x),&  z\ge x\\
U(z)m(z)+\int_z^x U(v)m(dv),&  z<x.
\end{cases}
\]
We have
\[
\frac{d}{dz}G_x(z)\ =\ 
\begin{cases}
U'(z)m(x),&  z\ge x,\\
U'(z)m(z),&  z< x,
\end{cases}
\]
and
\[
\frac{d^2}{dz^2}G_x(z)\ =\ 
\begin{cases}
U''(z)m(x),&  z\ge x,\\
U''(z)m(z)-2/\sigma^2(z),&  z< x,
\end{cases}
\]
which together with~\eqref{eq:U} implies that
\[
AG_x(z) =
\begin{cases}
-1,& z\le x,\\
0,& z>x, 	
\end{cases}
\]
and hence the function 
\begin{eqnarray}\label{G.dif}
f(z)\ =\ G_{h,x}(z) &:=& G_{x+h}(z)-G_x(z)
\end{eqnarray} 
solves~\eqref{eq:ode.diffusion}. 
Alternatively  one can notice that $U(x)$ corresponds to the scale function and $m(x)$ to the speed measure  and that (see~\cite[Chapter VII, Theorem 3.12]{RY1999}) 
\[
AG_x(z) = \frac{d}{dm(z)}\left(\frac{dG_x(z)}{-dU(z)}\right).
\]

Thus, if follows from~\eqref{G.dif} that for $y<x$,
\begin{eqnarray*}
H_y(x,x+h] &=& \int_x^{x+h} U(v)m(dv)\ =\
\int_x^{x+h}\frac{2 U(v)dv}{-U'(v)\sigma^2(v)}.
\end{eqnarray*}
More formally one can obtain the last equality 
from Corollary~3.8 and Exercise~3.20 in~\cite[Ch. VII.3]{RY1999}.

If the function $W(v):=U(v)/U'(v)\sigma^2(v)$ is long tailed 
at infinity---that is, for any fixed $u$, $W(v+u)\sim W(v)$
as $v\to\infty$---then 
we get the following local renewal theorem for $X_t$ starting at $y$,
\[
H_y(x,x+h]\ \sim \frac{2U(x)}{-U'(x)\sigma^2(x)}h\quad\mbox{as }x\to\infty.
\] 
Assume that
\begin{eqnarray}\label{assump.r.W}
2\mu(x)/\sigma^2(x) &\sim& r(x)\quad\mbox{as }x\to\infty,
\end{eqnarray} 
for some differentiable function $r(x)$ such that the quotient $r'(x)/r^2(x)$
has a limit at infinity. Hence, we can apply L'Hospital's rule to obtain 
\begin{eqnarray*}
\lim_{x\to\infty}\frac{U(x)}{-U'(x)/r(x)} &=& 
\lim_{x\to\infty}\frac{U'(x)}{-U''(x)/r(x)+U'(x)r'(x)/r^2(x)}\\
&=& \frac{1}{1+\lim_{x\to\infty}r'(x)/r^2(x)}.
\end{eqnarray*}
Therefore, for any fixed $h>0$,
\begin{eqnarray*}
H_y(x,x+h] &\sim& \frac{2}{\sigma^2(x)r(x)}
\frac{1}{1+\lim_{y\to\infty}r'(y)/r^2(y)}h\quad\mbox{as }x\to\infty.
\end{eqnarray*}

Note that this asymptotics do not assume existence of the limit of
the variance $\sigma^2(x)$ at infinity, and that happens because of
very specific nature of diffusion processes compared to Markov chains.
In order to get a result for Markov chains with growing second truncated moment 
of jumps, one would definitely need to assume regular growth of that
moments at infinity.
It is also clear that convergence of Markov chains to a stable law 
will play a r\^ole then.
 
Similar to nearest neighbour Markov chains, 
in the following two examples we consider canonical drifts
where $r'(y)/r^2(y)$ has either negative or zero limit at infinity.

\begin{example}\label{ex:dif.h.asy.1x}
If $\mu(x)\sim\mu/x$ and $\sigma^2(x)\to\sigma^2>0$ as $x\to\infty$
with $2\mu>\sigma^2$, then \eqref{assump.r.W} is satisfied with 
$r(x)=2\mu/\sigma^2 x$, $r'(x)/r^2(x)\to -\sigma^2/2\mu$, and we get
\begin{eqnarray*}
H_y(x,x+h] &\sim& \frac{2h}{2\mu-\sigma^2}x\quad\mbox{as }x\to\infty.
\end{eqnarray*}
\end{example}

\begin{example}\label{ex:dif.h.asy.alpha}
If $\mu(x)\sim\mu/x^\alpha$, $\mu>0$, $\alpha\in(0,1)$, 
and $\sigma^2(x)\to\sigma^2>0$ as $x\to\infty$,
then~\eqref{assump.r.W} is satisfied with 
$r(x)=2\mu/\sigma^2 x^\alpha$, $r'(x)/r^2(x)\to 0$, and we get
\begin{eqnarray*}
H_y(x,x+h] &\sim& \frac{h}{\mu}x^\alpha
\ \sim\ \frac{h}{\mu(x)}\quad\mbox{as }x\to\infty.
\end{eqnarray*}
\end{example}

\section{Preliminary bounds for renewal measure on growing intervals}\label{sec:ilrtgi}

Let $h(x)$ be an unboundedly growing function.
This section is mostly devoted to the construction of functions $G^*_{h,x}(y)$
and $G^{**}_{h,x}(y)$ such that the processes $G^*_{h,x}(X_n)$ 
and $G^{**}_{h,x}(X_n)$ have drifts, roughly speaking,
not less and not greater than the limiting jump variance times 
$\I\{y\in[x,x+h(x)]\}$ respectively. That allows us to conclude 
upper and lower bounds for the renewal measure on the interval $[x,x+h(x)]$
of growing length.

Let $r(x)$ be a decreasing differentiable function on $[0,\infty)$
satisfying the condition
\begin{eqnarray}\label{r.prime}
r'(x) &=& O(r^2(x))\quad\mbox{as }x\to\infty,
\end{eqnarray}
in the sequel $r(x)$ approximates the quotient 
$2m_1^{[s(x)]}(x)/m_2^{[s(x)]}(x)$. 
We shall impose assumptions on the truncated moments of Markov chains,
and doing that we always assume that the truncation function $s(x)$ 
increases and satisfies 
\[
s(x)=o\left(1/r(x)\right)\quad\mbox{as }x\to\infty.
\]
Define
\begin{align}\label{eq_u_x}
R(z)&:=\int_0^z r(y) dy,\qquad
U(x)\ :=\ \int_x^\infty e^{-R(z)} dz,
\end{align}
where $U(x)$ is assumed finite,
compare to $U$ defined in \eqref{def:U.dif}. Clearly, 
\begin{equation*}
\frac{U''(y)}{U'(y)} = -r(y).
\end{equation*}
Let us fix an increasing function $s(x)$ of order $o(1/r(x))$ as $x\to\infty$.
Due to \eqref{r.prime},
\begin{equation}\label{eq.r.insensitivity}
r(x+y)\sim r(x),\quad R(x+y)-R(x)\to 0
\quad\text{and}\quad e^{-R(x+y)}\sim e^{-R(x)}
\end{equation}
as $x\to\infty$ uniformly for $|y|\le s(x)$. Also,
\begin{equation}\label{U3}
U'''(x)=(r^2(x)-r'(x))e^{-R(x)}=O\bigl(r^2(x)e^{-R(x)}\bigr) 
\end{equation}
and, consequently, 
\begin{eqnarray}\label{U.3.uni}
U'''(x+y) &=& O\bigl(r^2(x)e^{-R(x)}\bigr)\quad\mbox{as }x\to\infty
\mbox{ uniformly for }|y|\le s(x).
\end{eqnarray}

Let 
$$
G(y)\ :=\ U(0)-U(y)\ =\ \int_0^y e^{-R(z)} dz.
$$
We start with a result showing that $G(X_n)$
is almost a martingale provided the quotient $2m_1^{[s(x)]}(x)/m_2^{[s(x)]}(x)$
is asymptotically proportional to $r(x)$. 
 
\begin{lemma}\label{lem:upper.U}
Let $\theta(y)$ be a non-negative bounded function. Let
\begin{eqnarray}\label{sigma_uniform2}
\E\{|\xi(y)|^3;\ |\xi(y)|\le s(y)\} &=&
o\bigl(m_2^{[s(y)]}(y)\theta(y)/r(y)\bigr)\quad\mbox{as }y\to\infty.
\end{eqnarray} 
{\rm (i)} If
\begin{eqnarray}\label{cond.tail.left}
\P\{\xi(y)<-s(y)\} &=& 0\quad\mbox{for all }y\ge 0,
\end{eqnarray}
and
\begin{equation}\label{m1_below}
\frac{2m_1^{[s(y)]}(y)}{m_2^{[s(y)]}(y)} \ge (1+\theta(y))r(y) 
\quad\mbox{for all sufficiently large }y,
\end{equation}  
then there exists a $y^*>0$ such that 
\begin{eqnarray*}	
\E\{G(y+\xi(y))-G(y);\ \xi(y)\le s(y)\} &\ge& 0\quad\mbox{for all }y>y^*.
\end{eqnarray*}
{\rm (ii)} If
\begin{eqnarray}\label{cond.tail.right}
\P\{\xi(y)>s(y)\} &=& 0\quad\mbox{for all }y\ge 0,
\end{eqnarray}
and
\begin{equation}\label{m1_above}
\frac{2m_1^{[s(y)]}(y)}{m_2^{[s(y)]}(y)} \le (1-\theta(y))r(y) 
\quad\mbox{for all sufficiently large }y,
\end{equation}  
then there exists a $y^*>0$ such that 
\begin{eqnarray*}	
\E\{G(y+\xi(y))-G(y);\ \xi(y)\ge -s(y)\} &\le& 0\quad\mbox{for all }y>y^*.
\end{eqnarray*}
\end{lemma}	

\begin{proof}
(i) Since the function $G(y)$ is increasing,
\begin{eqnarray*}
\E G(y+\xi(y))-G(y)
&\ge& \E\{G(y+\xi(y))-G(y);\ |\xi(y)|\le s(y)\},
\end{eqnarray*}
due to the condition \eqref{cond.tail.left}.
Since $G'(y)=e^{-R(y)}$, $G''(y)=-r(y)e^{-R(y)}$, and
$G'''(y+z)=O(r^2(y))e^{-R(y)}$ as $y\to\infty$ uniformly for all $|z|\le s(y)$
due to the upper bound \eqref{U.3.uni} on $U'''$ and \eqref{eq.r.insensitivity},
application of Taylor's expansion up to the third derivative yields that,
for some $\gamma=\gamma(x,\xi(x))\in[0,1]$,
\begin{eqnarray*}
\lefteqn{\E\{G(y+\xi(y))-G(y);\ |\xi(y)|\le s(y)\}}\\
&=& m_1^{[s(y)]}(y)G'(y)+\frac12 m_2^{[s(y)]}(y)G''(y)\\
&&\hspace{30mm}+ \frac16 \E\{\xi^3(y)G'''(y+\gamma\xi(y));\ |\xi(y)|\le s(y)\}\\
&=& m_1^{[s(y)]}(y)e^{-R(y)}
-\frac12 m_2^{[s(y)]}(y)r(y)e^{-R(y)}\\
&&\hspace{30mm}+ O\Bigl(r^2(y)e^{-R(y)} 
\E\{|\xi^3(y)|;\ |\xi(y)|\le s(y)\}\Bigr)\quad\mbox{as }y\to\infty.
\end{eqnarray*}
The sum of the first two terms on the right hand side equals
\begin{eqnarray*}
\frac12 e^{-R(y)}\bigl(2m_1^{[s(y)]}(y)-m_2^{[s(y)]}(y) r(y)\bigr)
&\ge& \frac12 e^{-R(y)} m_2^{[s(y)]}(y)\theta(y)r(y),
\end{eqnarray*}
due to the condition \eqref{m1_below}. 
The third term on the right hand side of the previous equation
is of order $o\bigl(m_2^{[s(y)]}(y)\theta(y)r(y)e^{-R(y)}\bigr)$
owing to the condition \eqref{sigma_uniform2}.
These observations conclude the proof of (i).

(ii) Since the function $G(y)$ is increasing,
\begin{eqnarray*}
\E G(y+\xi(y))-G(y)
&\le& \E\{G(y+\xi(y))-G(y);\ |\xi(y)|\le s(y)\},
\end{eqnarray*}
due to the condition \eqref{cond.tail.right}.
The rest of the proof is very similar to part (i).
\end{proof}

\subsection{Upper bound}
\label{sec:upper}

Our derivation of an upper bound for the renewal measure of $X_n$ is based 
on the Lyapunov function $G^{**}_{h,x}(y)$ defined below in \eqref{eq_gh}.

For any $x$ and $h>0$, consider a piecewise differentiable function
\begin{equation}\label{eq_first_derivative_gh}
g^{**}_{h,x}(y) := 
\begin{cases}
0,& y\le x,\\
2(y-x),& y\in (x,x+h],\\
2h,& y\in (x+h,x+h+s(x+h)],\\
2he^{R(x+h+s(x+h))-R(y)},& y>x+h+s(x+h),
\end{cases}
\end{equation}
whose derivative satisfies
\begin{equation}\label{eq_first_derivative_gh.der}
g_{h,x}^{**\prime}(y)\ =\ 2\I\{y\in[x,x+h]\}\quad\mbox{for all }y< x+h+s(x+h),
\ y\not= x,\ x+h.
\end{equation}
Its integral---the function which originates from
the key function \eqref{G.dif} for diffusion processes,
\begin{equation}\label{eq_gh}
G^{**}_{h,x}(y) := \int_0^y g^{**}_{h,x}(z)dz,
\end{equation}
is an increasing bounded function, $G^{**}_{h,x}(\infty)<\infty$, because 
\begin{equation}\label{g**.le.G}
g^{**}_{h,x}(y)\ \le\ 2he^{R(x+h+s(x+h))-R(y)}\quad\mbox{for all }y,
\end{equation}
and hence,
\begin{eqnarray}\label{G.upper.at.infty}
G^{**}_{h,x}(\infty) &\le& 2h \int_x^\infty e^{R(x+h+s(x+h))-R(y)}dy\nonumber\\
&=& 2h e^{R(x+h+s(x+h))}U(x)\nonumber\\
&\le& 2h U(x) e^{R(x+h)+r(x+h)s(x+h)},
\end{eqnarray}
because $R$ is concave. As $s(x)=o(1/r(x))$, 
\begin{eqnarray}\label{G.upper.at.infty.ap}
G^{**}_{h,x}(\infty) &\le& 2h U(x) e^{R(x+h)+o(1)}\nonumber\\
&\le& 2h U(x) e^{R(x)+o(1)}\quad\mbox{as }x\to\infty,
\end{eqnarray}
for $h\le s(x)$, due to \eqref{eq.r.insensitivity}.

The function $G^{**}_{h,x}(y)$ is convex for $y\le x+h$.
For $y>x+h$, the function $G^{**}_{h,x}(y)$ increases
in a concave way with slope $2h$ at point $x+h$. 
Notice that, for $y>x+h+s(x+h)$ and $z>0$,
\begin{eqnarray*}
G^{**}_{h,x}(y+z)-G^{**}_{h,x}(y) &=& 2he^{R(x+h+s(x+h))}(G(y+z)-G(y))
\end{eqnarray*}
and, due to \eqref{g**.le.G}, for $y>x+h+s(x+h)$ and $z\le 0$, 
\begin{eqnarray*}
G^{**}_{h,x}(y+z)-G^{**}_{h,x}(y) &\ge& 2he^{R(x+h+s(x+h))}(G(y+z)-G(y)).
\end{eqnarray*}
Therefore, for all $y>x+h+s(x+h)$ and $z\in\R$
\begin{eqnarray}\label{lower.G.1}
G^{**}_{h,x}(y+z)-G^{**}_{h,x}(y) &\ge& 2he^{R(x+h+s(x+h))}(G(y+z)-G(y)).
\end{eqnarray}
Further, for $y\in(x+h,x+h+s(x+h)]$,
\begin{equation*}
g^{**}_{h,x}(y+z)\ \ge\ 2he^{R(y)-R(y+z)}\quad\mbox{for }z>0,
\end{equation*}
and
\begin{equation*}
g^{**}_{h,x}(y+z)\ \le\ 2he^{R(y)-R(y+z)}\quad\mbox{for }z\le 0.
\end{equation*}
Therefore, for $y\in(x+h,x+h+s(x+h)]$, 
\begin{eqnarray}\label{lower.G.2}
G^{**}_{h,x}(y+z)-G^{**}_{h,x}(y) &\ge& 2he^{R(y)}(G(y+z)-G(y)).
\end{eqnarray}

\begin{lemma}\label{lem:upper}
Assume that the conditions \eqref{sigma_uniform2}--\eqref{m1_below} hold.
Then there exists an $x^*>0$ such that, 
for all $x>x^*$, $y\in\R$, $h\le s(x)$, and $t\in(0,h/2)$,
\begin{eqnarray}\label{eq_gh_super}	
\E G^{**}_{h,x}(y+\xi(y))-G^{**}_{h,x}(y)
&\ge& m_2^{[t]}(y)\I\{y\in [x+t,x+h-t]\}.
\end{eqnarray}
\end{lemma}	

\begin{proof}
Since the function $G^{**}_{h,x}(y)$ is zero for $y\le x$ and positive for $y>x$, 
the mean drift of $G^{**}_{h,x}$ is non-negative for all $y\le x$ and  
the inequality \eqref{eq_gh_super} follows for this range of $y$. 

Since $G^{**}_{h,x}(y)$ is increasing and due to \eqref{cond.tail.left},
\begin{eqnarray*}
\E G^{**}_{h,x}(y+\xi(y))-G^{**}_{h,x}(y)
&\ge& \E\{G^{**}_{h,x}(y+\xi(y))-G^{**}_{h,x}(y);\ |\xi(y)|\le s(y)\}
\ =:\ E.
\end{eqnarray*}
Positiveness of $E$ for $y>x+h$ follows from \eqref{lower.G.1} and
\eqref{lower.G.2}, by Lemma \ref{lem:upper.U}.

Thus, it remains to estimate $E$ from below for $y\in[x,x+h]$.
By Taylor's expansion for $G^{**}_{h,x}$ with integral remainder term,
\begin{eqnarray}\label{E.Taylor}	
E &=& m_1^{[s(y)]}(y)g^{**}_{h,x}(y)
+\E\Bigl\{\int_y^{y+\xi(y)}g^{**\prime}_{h,x}(z)(y+\xi(y)-z)dz;\ |\xi(y)|\le s(y)\Bigr\}.
\nonumber\\[-2mm]
\end{eqnarray}
Since $g^{**}_{h,x}(z)\ge 0$ and $g_{h,x}^{**\prime}(z)\ge 0$ 
for all $z\in[0,x+h+s(x+h)]$,
we obtain for all sufficiently large $x$ and $y\in[x,x+h]$
\begin{eqnarray*}	
    E &\ge& \E\Bigl\{\int_y^{y+\xi(y)}g^{**\prime}_{h,x}(z)(y+\xi(y)-z)dz;\ |\xi(y)|\le t\Bigr\}\\
&\ge& 2\I\{y\in [x+t,x+h-t]\}
\E\Bigl\{\int_y^{y+\xi(y)}(y+\xi(y)-z)dz;\ |\xi(y)|\le t\Bigr\}\\
&=& m_2^{[t]}(y)\I\{y\in [x+t,x+h-t]\},
\end{eqnarray*}
because $g_{h,x}^{**\prime}(z)=2$ for all $z\in(x,x+h]$ which concludes the proof.
\end{proof}

\begin{proposition}\label{thm:renewal.ub}
Assume that conditions of Lemma~\ref{lem:upper} hold. 
Then there exists an $x^*>0$ such that, 
for all $x>x^*$, $h\le s(x)$, and $t\in(0,h/2)$,
\begin{eqnarray*}
H(x+t,x+h-t] &\le& \frac{G^{**}_{h,x}(\infty) - \E G^{**}_{h,x}(X_0)}
{\min_{y\in[x+t,x+h-t]} m_2^{[t]}(y)}.
\end{eqnarray*}
\end{proposition}

\begin{proof}
Consider the following decomposition 
$$
G^{**}_{h,x}(X_n)
=\sum_{k=0}^{n-1} (G^{**}_{h,x}(X_{k+1})-G^{**}_{h,x}(X_k))+ G^{**}_{h,x}(X_0).
$$
Since $G^{**}_{h,x}(y)$ is bounded by $G^{**}_{h,x}(\infty)$, we obtain 
\begin{eqnarray*}
G^{**}_{h,x}(\infty) &\ge& \E G^{**}_{h,x}(X_n)\\
&=& \E G^{**}_{h,x}(X_0)+\sum_{k=0}^{n-1} 
\E [G^{**}_{h,x}(X_{k+1})-G^{**}_{h,x}(X_k)]\\
&\ge& \E G^{**}_{h,x}(X_0)+\sum_{k=0}^{n-1} 
\E \{m_2^{[t]}(X_k);X_k\in(x+t,x+h-t]\},
\end{eqnarray*}
for $x>x_*$, by Lemma \ref{lem:upper}. Hence, for any $n$, 
\begin{eqnarray*}
\sum_{k=0}^{n-1} \P\{X_k \in (x+t,x+h-t]\}
&\le& \frac{G^{**}_{h,x}(\infty)-\E G^{**}_{h,x}(X_0)}
{\min_{y\in[x+t,x+h-t]}m_2^{[t]}(y)}.
\end{eqnarray*}
Letting $n$ to infinity we arrive at the conclusion. 
\end{proof}

\subsection{Lower bound}
\label{sec:lower}

We now turn to an accompanying lower bound for the renewal measure.
To this end we consider a differentiable function
\begin{equation}\label{eq_first_derivative_gh*}
g^*_{h,x}(y) := 
\begin{cases}
0,& y\le x,\\
2(y-x),& y\in (x,x+h],\\
2he^{R(x+h)-R(y)},& y>x+h,
\end{cases}
\end{equation}
whose derivative satisfies
\begin{equation}\label{eq_first_derivative_gh.der*}
g_{h,x}^{*\prime}(y)\ \le\ 2\I\{y\in[x,x+h]\}\quad\mbox{for all }y\ge 0.
\end{equation}
Its integral---which similar to \eqref{eq_gh} originates from
the key function \eqref{G.dif} for diffusion processes,
\begin{equation}\label{eq_gh*}
G^*_{h,x}(y) := \int_0^y g^*_{h,x}(z)dz,
\end{equation}
is an increasing bounded function, $G^*_{h,x}(\infty)<\infty$, and
\begin{eqnarray}\label{G.upper.at.infty*}
G^*_{h,x}(\infty) &=& h^2+2h e^{R(x+h)}U(x+h)\nonumber\\
&\ge& 2h e^{R(x)}U(x+h).
\end{eqnarray}
For $h\le s(x)=o(1/r(x))$,
\begin{eqnarray}\label{G.upper.at.infty.ap*}
G^*_{h,x}(\infty) &\ge& (2+o(1))h e^{R(x)}U(x)\quad\mbox{as }x\to\infty.
\end{eqnarray}

Also define a concave function
\begin{equation}\label{eq_gh**}
G^{*<}_{h,x}(y) := h^2+2he^{R(x+h)}\int_{x+h}^y e^{-R(z)}dz.
\end{equation}
Observe the inequality
\begin{equation}\label{G*ge**}
G^*_{h,x}(y) \ge G^{*<}_{h,x}(y)\quad\mbox{for all }y\le x+h,
\end{equation}
and equality
\begin{equation}\label{G*=**}
G^*_{h,x}(y) = G^{*<}_{h,x}(y)\quad\mbox{for all }y>x+h,
\end{equation}
Hence, for $y>x+h$ and $z>0$,
\begin{eqnarray}\label{lower.G.1*}
G^*_{h,x}(y-z)-G^{*<}_{h,x}(y-z) &\le& 
G^*_{h,x}(y)-G^{*<}_{h,x}(y-z)\nonumber\\
&=& G^{*<}_{h,x}(y)-G^{*<}_{h,x}(y-z)\nonumber\\
&=& 2he^{R(x+h)}(G(y)-G(y-z)).
\end{eqnarray}

\begin{lemma}\label{thm:renewal.lb}
Assume that the conditions \eqref{sigma_uniform2},
\eqref{cond.tail.right} and \eqref{m1_above} hold. 
Then there exists an $x^*>0$ such that,
for all $x>x^*$, $y\ge 0$, $h\le s(x)$, and $t\in(0,h/2)$,
\begin{multline*}
\E G^*_{h,x}(y+\xi(y))-G^*_{h,x}(y)\\
\le \begin{cases}
0,& y\le x-s(x),\\
2h \E\{\xi(y);\xi(y)\in(x-y, s(y))\},& y\in(x-s(x),x-t],\\
(1+hr(y)) m_2^{[s(y)]}(y),& y\in(x-t,x+h+t],\\
3h\E\{|\xi(y)|; -s(y)<\xi(y)<x+h-y\},& y>x+h+t.
\end{cases}	
\end{multline*}
\end{lemma}

\begin{proof}
Since $G^*_{h,x}(y)$ is increasing in $y$, we obtain 
\begin{eqnarray*}
\E G^*_{h,x}(y+\xi(y))-G^*_{h,x}(y) &\le&
\E\{G^*_{h,x}(y+\xi(y))-G^*_{h,x}(y);\ \xi(y)\ge -s(y)\}\nonumber\\
&=& \E\{G^*_{h,x}(y+\xi(y))-G^*_{h,x}(y);\ |\xi(y)|\le s(y)\}
\ =:\ E,
\end{eqnarray*}
due to \eqref{cond.tail.right}.

\underline  {Case $y\le x-t$.} 
It follows from the definition of $G^*_{h,x}$ that
$G^*_{h,x}(x+z)\le 2hz$ for all $z>0$ which yields
$G^*_{h,x}(y+z)\le 2h(y-x+z)$ for all $y\le x$ and $z>0$. Therefore,
\begin{eqnarray}\label{case1}
E &\le& 2h \E\left\{\xi(y);\xi(y)\in(x-y, s(y)]\right\}, 
\end{eqnarray}
and the conclusion of the lemma follows for $y\le x-t$. 

\underline{Case $y\in (x-t,x+h+t]$.} 
We proceed similarly to Lemma~\ref{lem:upper}. 
By Taylor's expansion \eqref{E.Taylor}, 
\begin{eqnarray*}
E &\le& m_1^{[s(y)]}(y)g^*_{h,x}(y)+m_2^{[s(y)]}(y)\\
&\le& \frac12m_2^{[s(y)]}(y)r(y)g^*_{h,x}(y)+m_2^{[s(y)]}(y)\\
&\le& m_2^{[s(y)]}(y)(hr(y)+1),
\end{eqnarray*}
due to \eqref{m1_above}, \eqref{eq_first_derivative_gh.der*}
and inequality $g^*_{h,x}(y)\le 2h$, for all sufficiently large $y$.
Thus the conclusion of the lemma follows for $y\in (x-t,x+h+t]$. 

\underline{Case $y>x+h+t$.} Since the function $G(y)$ is concave,
\begin{eqnarray*}
G(y-z)-G(y) &\le& zG'(y-z)\ =\ ze^{-R(y-z)}\quad\mbox{for all }z>0. 
\end{eqnarray*}
Therefore, as $y\to\infty$,
\begin{eqnarray*}
G(y-z)-G(y) &\le& ze^{-R(y)}(1+o(1))\quad\mbox{uniformly for all }z\in[0,s(y)]. 
\end{eqnarray*}
Thus it follows from \eqref{lower.G.1*} that, as $y\to\infty$,
\begin{eqnarray}\label{G*-**}
G^*_{h,x}(y-z)-G^{*<}_{h,x}(y-z) &\le& 2hze^{R(x+h)-R(y)}(1+o(1))\nonumber\\
&\le& 2hz(1+o(1))\quad\mbox{uniformly for all }h,z\in[0,s(y)]. \nonumber\\[-1mm]
\end{eqnarray}
The inequality \eqref{G*ge**} and equality \eqref{G*=**} 
allow us to conclude that, for $y>x+h$,
\begin{eqnarray*}
E &=& \E\{G^{*<}_{h,x}(y+\xi(y))-G^{*<}_{h,x}(y);\ |\xi(y)|\le s(y)\}\\
&&+\E\{G^*_{h,x}(y+\xi(y))-G^{*<}_{h,x}(y+\xi(y));\ |\xi(y)|\le s(y)\}\\
&=& \E\{G^{*<}_{h,x}(y+\xi(y))-G^{*<}_{h,x}(y);\ |\xi(y)|\le s(y)\}\\
&&+\E\{G^*_{h,x}(y+\xi(y))-G^{*<}_{h,x}(y+\xi(y));\ \xi(y)\in[-s(y),x+h-y]\}\\
&\le& \E\{G^*_{h,x}(y+\xi(y))-G^{*<}_{h,x}(y+\xi(y));\ \xi(y)\in[-s(y),x+h-y]\},
\end{eqnarray*}
by the second statement of Lemma \ref{lem:upper.U}.
Applying here \eqref{G*-**} we deduce, for all sufficiently large $x$ and $y>x+h$,
\begin{eqnarray*}
E &\le& 3h \E\{|\xi(y)|;\ \xi(y)\in[-s(y),x+h-y]\}.
\end{eqnarray*}
Combining altogether we conclude the result of the lemma for $y>x+h+t$. 
\end{proof}

\begin{proposition}\label{prop5}
Let the assumptions of Lemma~\ref{thm:renewal.lb} hold. 
Then there exists an $x^*>0$ such that, 
for all $x>x^*$, $y\ge 0$, $h\le s(x)$, and $t\in(0,h/2)$,
\begin{eqnarray*}
H(x-t,x+h+t] &\ge& 
\frac{G^*_{h,x}(\infty)-\E G^*_{h,x}(X_0)-\delta(x)}
{\max_{y\in[x-t,x+h+t]}(1+hr(y))m_2^{[s(y)]}(y)},
\end{eqnarray*}
where 
\begin{eqnarray*}
\delta(x) &=& 2h \int_{x-s(x)}^{x-t} H(dy) \E\{\xi(y);x-y<\xi(y)<s(y)\}\\
&&\hspace{5mm}+3h \int_{x+h+t}^\infty H(dy) \E\{|\xi(y)|;-s(y)<\xi(y)<x+h-y\}. 
\end{eqnarray*}
\end{proposition}	

\begin{proof}
Consider the decomposition 
$$
G^*_{h,x}(X_n)
=\sum_{k=0}^{n-1} (G^*_{h,x}(X_{k+1})-G^*_{h,x}(X_k))+ G^*_{h,x}(X_0).
$$
We deduce from Lemma \ref{thm:renewal.lb} that, 
for some $c<\infty$ and all $x>x_*$,
\begin{eqnarray*}
\lefteqn{\E G^*_{h,x}(X_n)}\\
&=& \E G^*_{h,x}(X_0)
+\sum_{k=0}^{n-1} \E(G^*_{h,x}(X_{k+1})-G^*_{h,x}(X_k))\\
&\le& \E G^*_{h,x}(X_0)+ \sum_{k=0}^{n-1} 
\E\left\{(1+hr(X_k))m_2^{[s(X_k)]}(X_k);X_k \in (x-t,x+h+t]\right\}\\ 
&&+2h \sum_{k=0}^{n-1} \int_{x-s(x)}^{x-t} \P\{X_k \in dy\} 
\E\{\xi(y);x-y<\xi(y)<s(y)\}\\  
&&+3h\sum_{k=0}^{n-1} \int_{x+h+t}^\infty
\P\{X_k \in dy\} \E\{|\xi(y)|;-s(y)<\xi(y)<x+h-y\}.
\end{eqnarray*}
Hence, for any $n$, 
\begin{eqnarray*}
\sum_{k=0}^{n-1} \P\{X_k \in (x-t,x+h+t]\} &\ge& 
\frac{\E G^*_{h,x}(X_n)-\E G^*_{h,x}(X_0)-\delta(x)}
{\max_{y\in[x-t,x+h+t]}(1+hr(y))m_2^{[s(y)]}(y)}.
\end{eqnarray*}
Letting $n$ to infinity we arrive at the conclusion
due to the convergence $G^*_{h,x}(X_n)\to G^*_{h,x}(\infty)$
which in its turn follows from Lemma~\ref{lem:upper} together with 
the martingale convergence theorem and the assumption~\eqref{eq:irreducibility}.  
\end{proof}

In order to get a lower bound in a closed form, 
we need to derive conditions under which the term $\delta(x)$ 
in Proposition~\ref{prop5} is of order $o(G^*_{h,x}(\infty))$ as $x\to\infty$.
In the next result we demonstrate how to bound $\delta(x)$ 
provided an appropriate upper bound for the renewal measure is available.

\begin{lemma}\label{thm:renewal.ub.lower}
Let
\begin{eqnarray}\label{general_upper}
H(x+t,x+h-t] &\le& C_1h U(x)e^{R(x)}\quad\mbox{for some }C_1<\infty,
\end{eqnarray}
and, for some random variable $\xi$ with $\E\xi^2<\infty$,
\begin{eqnarray}\label{majorant_third.lower}
|\xi(y)| &\le_{st}& \xi\quad\mbox{for all }y\ge 0.
\end{eqnarray}
Then $\delta(x) = o\bigl(hU(x)e^{R(x)}\bigr)$ as $x\to\infty$.
\end{lemma}

\begin{proof}
Let us analyse the first term in $\delta(x)$.
The stochastic majorisation condition \eqref{majorant_third.lower} yields that 
\begin{eqnarray*}
\int_{x-s(x)}^{x-t} H(dy) \E\{\xi(y);\ x-y<\xi(y)<s(y)\} 
&\le& \int_{x-s(x)}^{x-t} H(dy) \E\{\xi;\ \xi>x-y\}.
\end{eqnarray*}
Further, using the upper bound \eqref{general_upper} 
applied to $h(x)=3t$ we deduce 
\begin{eqnarray*}
\int_{x-s(x)}^{x-t(x)} H(dy) \E\{\xi;\ \xi>x-y\} 
&\le& \sum_{n=1}^{s(x)/t}	H(x-(n+1)t,x-nt] \E\{\xi;\ \xi>nt\}\\ 
&\le& C_2tU^*(x)e^{R^*(x)} \sum_{n=1}^{s(x)/t} \E\{\xi;\ \xi>nt\} \\ 
&\le& C_2tU^*(x)e^{R^*(x)} \E\{\xi^2/t;\ \xi>t\}\\
&=& o(U^*(x)e^{R^*(x)})\quad\mbox{as }t,\ x\to\infty,
\end{eqnarray*}
by the condition $\E\xi^2<\infty$. Hence the first term in $\delta(x)$ 
is of order $o\bigl(h(x)U^*(x)e^{R^*(x)}\bigr)$ as required.
The second term in $\delta(x)$ is of the same order,
as follows by the same arguments, and we conclude the proof.
\end{proof}

\section{On two Markov chains with asymptotically equal jumps}

In this section, we prove a coupling that allows us to compare
two Markov chains which have asymptotically equal jumps.
The following result is repeatedly used in the sequel each time we want
to simplify our calculations related to the characteristics of $X_n$.
We formulate this result in the following general setting.

Let $Y_n$ and $Z_n$ be two Markov chains with jumps
$\eta(x)$ and $\zeta(x)$ respectively.
Denote by $H_y^Y$ the renewal measure
generated by the chain $Y_n$ with initial state $Y_0=y$, that is,
\begin{eqnarray*}
H_y^Y(A) &:=& \sum_{n=0}^\infty \P_y\{Y_n\in A\},\quad A\in\mathcal B(\R).
\end{eqnarray*}

\begin{lemma}\label{l:XY.equiv}
Let the random variables $\eta(x)$ and $\zeta(x)$ 
be constructed on the same probability space in such a way that
\begin{eqnarray}\label{rec.3.1.hy}
\P\{\eta(x)\not=\zeta(x)\} &\le& p(x)v(x)\quad\mbox{for all }x,
\end{eqnarray}
where $v(x)>0$ and $p(x)>0$ are decreasing functions
and $p(x)>0$ is integrable at infinity.
Let also, for some $c<\infty$,
\begin{eqnarray}\label{rec.3.1.hyz}
H_y^Y(x,2x] &\le& \frac{cx}{v(x)}
\quad\mbox{for all }y\mbox{ and }x.
\end{eqnarray}
Then, for any $\varepsilon>0$ there exists an $x_\varepsilon$ such that
the chains $Y_n$ and $Z_n$ may be constructed on the same probability
space in such a way that
\begin{eqnarray*}
\P\{Y_n=Z_n\mbox{ for all }n\ge 0\} &\ge& 1-\varepsilon
\end{eqnarray*}
provided $Z_0=Y_0>x_\varepsilon$.
\end{lemma}

\begin{proof}
It follows from the condition \eqref{rec.3.1.hyz} that, for all $z\in\R$,
\begin{eqnarray}\label{rec.3.1.sslln}
\P\{Y_n>z\mbox{ for all }n\ge0\mid Y_0=y\} &\to& 1\quad\mbox{as }y\to\infty.
\end{eqnarray}

Let us construct a probability space and two sequences of independent
random fields $\{\eta_n(x),x\in\R\}_{n\ge 0}$ and
$\{\zeta_n(x),x\in\R\}_{n\ge 0}$ on this space such that
\begin{eqnarray}\label{rec.3.1.hy.n}
\P\{\eta_n(x)\not=\zeta_n(x)\} &\le& p(x)v(x)
\quad\mbox{for all }x\in\R\mbox{ and }n\ge 0,
\end{eqnarray}
which is possible due to \eqref{rec.3.1.hy}.
Then let us define Markov chains as follows:
\begin{eqnarray*}
Y_{n+1}\ =\ Y_n+\eta_{n+1}(Y_n), && Z_{n+1}\ =\ Z_n+\zeta_{n+1}(Z_n),
\end{eqnarray*}
Fix an $\varepsilon>0$. For any $z$,
\begin{eqnarray*}
\lefteqn{\P\{Z_n\neq Y_n\mbox{ for some }n\mid Y_0=y\}}\\
&&\hspace{7mm}\le\ \P\{Y_n\le z\mbox{ for some }n\mid Y_0=y\}\\
&&\hspace{20mm}+\P\{Z_n\neq Y_n\mbox{ for some }n,
Y_n\ge z\mbox{ for all }n\mid Y_0=y\}.
\end{eqnarray*}
Owing to \eqref{rec.3.1.sslln}, there exists an $y_1(z)$ such that
\begin{eqnarray*}
\P\{Y_n\le z\mbox{ for some }n\mid Y_0=y\}
&\le& \varepsilon/2\quad\mbox{for all }y>y_1(z).
\end{eqnarray*}
Given $Z_0=Y_0>z$,
\begin{eqnarray*}
\lefteqn{\P\{Z_n\neq Y_n\mbox{ for some }n,\ Y_n>z\mbox{ for all }n\mid Y_0=y\}}\\
&&\hspace{5mm}\le\ \P\{\eta_{n+1}(Y_n)\not=\zeta_{n+1}(Z_n),\ Y_n=Z_n
\mbox{ for some }n,\ Y_n>z\mbox{ for all }n\mid Y_0=y\}.
\end{eqnarray*}
The probability on the right hand side does not exceed the following sum
\begin{eqnarray*}
&&\sum_{n=0}^\infty\P\{\eta_{n+1}(Y_n)\not=\zeta_{n+1}(Z_n), 
\ Z_n=Y_n>z\mid Y_0=y\}\\
&&\hspace{3cm}\le \int_z^\infty \P\{\eta(x)\not=\zeta(x)\}H^Y_y(dx)\\
&&\hspace{3cm}\le \int_z^\infty p(x)v(x) H^Y_y(dx),
\end{eqnarray*}
by the condition \eqref{rec.3.1.hy}.
The last integral tends to $0$ as $z\to\infty$.
Indeed, both functions $p(z)$ and $v(x)$ are decreasing, hence
\begin{eqnarray*}
\int_{2z}^\infty p(x)v(x) H^Y_y(dx) &\le&
\sum_{i=1}^\infty p(x_i)v(x_i) H^Y_y(x_i,x_{i+1}],
\end{eqnarray*}
where $x_i:=2^{i-1} z$ for $i\ge 0$.
Then, by the condition \eqref{rec.3.1.hyz} on $H_y^Y$,
\begin{eqnarray*}
\int_{2z}^\infty p(x)v(x) H^Y_y(dx)
&\le& c\sum_{i=1}^\infty p(x_i) x_i\\
&=& 2c\sum_{i=1}^\infty p(x_i)(x_i-x_{i-1}).
\end{eqnarray*}
Then decrease of the function $p(x)$ yields
\begin{eqnarray*}
\sum_{i=1}^\infty p(x_i)(x_i-x_{i-1})
&\le& \int_z^\infty p(u)du\ \to\ 0\quad\mbox{as }z\to\infty,
\end{eqnarray*}
because $p(x)$ is integrable. Hence,
\begin{eqnarray}\label{int.prH.fin}
\int_{2z}^\infty p(x)v(x) H^Y_y(dx) &\to& 0
\quad\mbox{as }z\to\infty\mbox{ uniformly for all }y,
\end{eqnarray}
which implies convergence to $0$ of the integral from $z$ to $\infty$.
Then the integral from $z$ to $\infty$ is less than $\varepsilon/2$
for a sufficiently large $z=z(\varepsilon)$ which
concludes the proof with $x_\varepsilon=y_1(z(\varepsilon))$.
\end{proof}

\begin{lemma}\label{l:XY.renew.equiv}
Let the conditions of Lemma \ref{l:XY.equiv} hold. If there exist
non-negative functions $h(x)$ and $g(x)$ such that
\begin{equation}
\label{equiv.1}
H^Y(x,x+h(x)]\sim g(x)\quad\mbox{as }x\to\infty
\end{equation}
for any distribution of $Y_0$ and
\begin{equation}\label{equiv.2}
\sup_y H_y^Y(x,x+h(x)]=O(g(x))\quad\mbox{as }x\to\infty,
\end{equation}
then, for any distribution of $Z_0$,
$$
H^Z(x,x+h(x)]\sim g(x)\quad\mbox{as }x\to\infty.
$$
\end{lemma}

\begin{proof}
Let us construct $\{\eta_n(x),x\in\R\}_{n\ge 0}$ and
$\{\zeta_n(x),x\in\R\}_{n\ge 0}$ as in \eqref{rec.3.1.hy.n}
and then the Markov chains $Y_n$ and $Z_n$ as there.

Fix an $\varepsilon>0$ and let $x_\varepsilon$ be delivered by 
the last lemma. Let $\tau:=\min\{n\ge 0:Z_n>x_\varepsilon\}$ and 
consider $Y_k$ with initial value $Y_0=Z_\tau$. Define 
$$
\mu:=\min\{k\ge 1:Y_k\not= Z_{\tau+k}\}.
$$
By Lemma \ref{l:XY.equiv}, $\P\{\mu<\infty\}\le\varepsilon$. 
For $x>x_\varepsilon$,
\begin{eqnarray*}
\lefteqn{\sup_y H_y^Z(x,x+h(x)]}\\
&\le& \sup_y \E_y\sum_{n=\tau}^{\tau+\mu-1} \I\{Z_n\in(x,x+h(x)]\}
+\sup_y \E_y\sum_{n=\tau+\mu}^\infty \I\{Z_n\in(x,x+h(x)]\}.
\end{eqnarray*}
The first expectation on the right hand side is not greater than
$H_y^Y(x,x+h(x)]$ because $Z_n=Y_{n-\tau}$ between $\tau$ and $\tau+\mu-1$.
The second one possesses the following upper bound
\begin{eqnarray*}
\E_y\sum_{n=\mu}^\infty \I\{Z_n\in(x,x+h(x)]\} &=&
\E_y\Bigl\{\sum_{n=\tau+\mu}^\infty \I\{Z_n\in(x,x+h(x)]\}
\Big| \mu<\infty\Bigr\}
\P\{\mu<\infty\}\\ 
&\le& \sup_z H_z^Z(x,x+h(x)]\varepsilon.
\end{eqnarray*}
Therefore,
\begin{eqnarray*}
\sup_y H_y^Z(x,x+h(x)] &\le& \frac{1}{1-\varepsilon} \sup_yH_y^Y(x,x+h(x)],
\end{eqnarray*}
which, due to assumption \eqref{equiv.2} implies that 
\begin{equation}
\label{equiv.3}
\sup_y H_y^Z(x,x+h(x)]=O(g(x)).
\end{equation}

For any distribution of $Z_0$ we have
\begin{eqnarray*}
\lefteqn{H^Z(x,x+h(x)]}\\
&=& \E\sum_{n=\tau}^{\tau+\mu-1} \I\{Z_n\in(x,x+h(x)]\}
+\E\sum_{n=\tau+\mu}^\infty \I\{Z_n\in(x,x+h(x)]\}\\
&=& \E\sum_{n=\tau}^{\tau+\mu-1} \I\{Y_n\in(x,x+h(x)]\}
+\E\sum_{n=\tau+\mu}^\infty \I\{Z_n\in(x,x+h(x)]\}\\
&=& \E H^Y_{Z_\tau}(x,x+h(x)]\\
& &\hspace{5mm}-\E\E_{Z_\tau}\sum_{n=\mu}^{\infty} \I\{Y_n\in(x,x+h(x)]\}
+\E\sum_{n=\tau+\mu}^\infty \I\{Z_n\in(x,x+h(x)]\}
\end{eqnarray*}
According to \eqref{equiv.1} and \eqref{equiv.2},
$\E H^Y_{Z_\tau}(x,x+h(x)]\sim g(x)$. 
Further, as we have seen in the first part of the proof, for all large enough $x$,
$$
\E\sum_{n=\tau+\mu}^\infty \I\{Z_n\in(x,x+h(x)]\}
\le\varepsilon \sup_y H_y^Z(x,x+h(x)].
$$
Letting $\varepsilon\to0$ and using \eqref{equiv.3}, we conclude that
$$
\E\sum_{n=\tau+\mu}^\infty \I\{Z_n\in(x,x+h(x)]\}=o(g(x))
\quad\mbox{as }x\to\infty.
$$
Thus, it remains to show that 
$$
\E\E_{Z_\tau}\sum_{n=\mu}^{\infty} \I\{Y_n\in(x,x+h(x)]\}=o(g(x))
\quad\mbox{as }x\to\infty.
$$
But this expectation can be bounded in the same manner:
\begin{eqnarray*}
\E_{Z_\tau}\sum_{n=\mu}^{\infty} \I\{Y_n\in(x,x+h(x)]\}
&\le& \E\P_{Z_\tau}(\mu<\infty)\sup_y H_y^Y(x,x+h(x)]\\
&\le& \varepsilon \sup_y H_y^Y(x,x+h(x)].
\end{eqnarray*}
Combining this with \eqref{equiv.2} we complete the proof.
\end{proof}

\section{Proofs of Theorems \ref{thm:regular}, \ref{thm:critical}
and \ref{thm:weibull}}

\begin{proof}[Proof of Theorem \ref{thm:regular}]
Consider a modified Markov chain $\widetilde X_n$ on the same probability 
space as $X_n$ with jumps $\widetilde\xi(x)$ defined as follows:
\begin{eqnarray}\label{def:Z.jumps}
\widetilde\xi(x) &=& \left\{
\begin{array}{ll}
\xi(x) &\mbox{if }|\xi(x)|\le s(x);\\
\mbox{any value} &\mbox{if }|\xi(x)|>s(x).
\end{array}
\right.
\end{eqnarray}
If $\widetilde X_n$ does not satisfy the weak irreducibility condition 
\eqref{eq:irreducibility}, then we can increase the value of $s(x)$
on some set bounded above in such a way that then $\widetilde X_n$ do satisfy 
\eqref{eq:irreducibility}. Indeed, it follows from the conditions
\eqref{m1.m2.1x}, \eqref{majorant_third} and \eqref{majorant_third_moment_exists}
that there exist a sufficiently high level $x_0$ and an $\varepsilon>0$
such that $\P\{\xi(x)\ge\varepsilon\}\ge\varepsilon$ for all $x\ge x_0$.
Then it suffices to increase $s(x)$ on the set $(-\infty,x_0]$
to ensure the condition \eqref{eq:irreducibility} for $\widetilde X_n$.

Without loss of generality we assume that $h(x)\le s(x)$.
Let us choose a function $t(x)\uparrow\infty$ of order $o(h(x))$ as $x\to\infty$.

Fix some $c>1$ and consider $r(x)=c/(1+x)$. Then, 
$$
R(x)=c\log(1+x)\quad \mbox{ and } \quad U(x)=(1+x)^{1-c}/(c-1).  
$$ 
Therefore, 
\begin{eqnarray}
\label{regular_G}
U(x)e^{R(x)}=\frac{x+1}{c-1}.  
\end{eqnarray}

The chain $\widetilde X_n$ satisfies the condition \eqref{cond.tail.left}.
Fix some $c^{**}\in(1,2\mu/b)$ and define $r^{**}(x)=c^{**}/(1+x)$,
which ensures the condition \eqref{m1_below} with $\theta=(2\mu/bc^{**}-1)/2>0$.
The condition \eqref{sigma_uniform2} is immediate from the upper bound
\begin{equation}\label{eq_third_simplify}
\E\{|\xi(y)|^3;\ |\xi(y)|\le s(y)\}\le s(y)m_2^{[s(y)]}(y)
\end{equation}
and the relation $s(y) = o(y)$.  Also,
\begin{eqnarray*}
m_2^{[t(x)]}(y) &\to& b\quad\mbox{as }x\to\infty,
\end{eqnarray*}	
by the conditions \eqref{majorant_third} and \eqref{majorant_third_moment_exists}. 
As a result, by Proposition~\ref{thm:renewal.ub}, as $x\to\infty$,
\begin{eqnarray*}
\widetilde H(x+t(x),x+h(x)-t(x)] &\le& 
\frac{G^{**}_{h,x}(\infty)}{b+o(1)}\\
&\le& \frac{2+o(1)}{(c^{**}-1)b}xh(x),
\end{eqnarray*} 
owing to \eqref{G.upper.at.infty.ap} and
\eqref{regular_G}. Letting $c^{**}\to2\mu/b$, we get
$$
\widetilde H(x+t(x),x+h(x)-t(x)]\le\frac{2+o(1)}{2\mu-b}xh(x)
\quad\mbox{as }x\to\infty.
$$
Taking into account that $t(x)=o(h(x))$ we conclude the following upper bound  
\begin{eqnarray}\label{regular_upper}
\widetilde H(x,x+h(x)] &\le& \frac{2+o(1)}{2\mu-b}xh(x)\quad\mbox{as }x\to \infty.
\end{eqnarray}

The chain $\widetilde X_n$ satisfies the condition \eqref{cond.tail.right}.
Fix some $c^*>2\mu/b$ and define $r^*(x)=c^*/(1+x)$,
which ensures the condition \eqref{m1_above} with $\theta=(1-2\mu/bc^*)/2>0$.
Then it follows from Proposition~\ref{prop5} that, as $x\to\infty$,
\begin{eqnarray*}
\widetilde H(x-t(x),x+h(x)+t(x)] &\ge& 
\frac{G^*_{h,x}(\infty)-\E G^*_{h,x}(X_0)-\delta(x)}{b+o(1)}\\
&\ge& (2+o(1))\frac{h(x)\frac{x}{c^*-1}-\delta(x)}{b+o(1)},
\end{eqnarray*}
due to \eqref{G.upper.at.infty.ap*} and \eqref{regular_G}. By the condition \eqref{majorant_third},
the chain $\widetilde X_n$ satisfies \eqref{majorant_third.lower}
which together with the upper bound \eqref{regular_upper} for the renewal
measure generated by $\widetilde X_n$ yields the upper bound 
for $\delta(x)$ delivered by Lemma \ref{thm:renewal.ub.lower}. Therefore,
\begin{eqnarray*}
\widetilde H(x-t(x),x+h(x)+t(x)] &\ge& \frac{2+o(1)}{(c^*-1)b}xh(x).
\end{eqnarray*}
owing to \eqref{regular_G}.
Letting here $c^*\to 2\mu/b$ and since $t(x)=o(h(x))$, we finally get
\begin{eqnarray*}
\widetilde H(x,x+h(x)] &\ge& \frac{2+o(1)}{2\mu-b}xh(x)\quad\mbox{as }x\to\infty.
\end{eqnarray*}
Combining this lower bound with the upper bound \eqref{regular_upper},
we conclude that
\begin{eqnarray*}
\widetilde H(x,x+h(x)] &\sim& \frac{2}{2\mu-b}xh(x)\quad\mbox{as }x\to\infty.
\end{eqnarray*}
Together with the condition \eqref{regular_left_tail} this allows us 
to apply Lemma \ref{l:XY.renew.equiv} to the two Markov chains, $Z_n=X_n$ 
and $Y_n=\widetilde X_n$, hence the same asymptotics
for the renewal measure generated by $X_n$.
\end{proof}

\begin{proof}[Proof of Theorem \ref{thm:critical}]
As in the proof of Theorem \ref{thm:regular}, 
from the very beginning we may assume
that $|\xi(y)|\le s(y)$ for all $y$ which implies both 
\eqref{cond.tail.left} and \eqref{cond.tail.right}.
Without loss of generality we assume that $h(x)\le s(x)$.

Fix $c>1$ and consider
\begin{eqnarray*}
r(x) &=& \frac{1}{x+e_{(m)}}+\frac{1}{(x+e_{(m)})\log (x+e_{(m)})}\\
&&+\ldots+\frac{c}{(x+e_{(m)})\log (x+e_{(m)})\ldots\log_{(m)}(x+e_{(m)})},
\end{eqnarray*}
where $e_{(m)}>0$ is defined by $\log_{(m)} e_{(m)}=1$. Therefore,
\begin{eqnarray*}
R(x)&=&\log(x+e_{(m)})+\log\log(x+e_{(m)})\\
&&+\ldots+\log_{(m)}(x+e_{(m)})+c\log_{(m+1)}(x+e_{(m)})-C_m
\end{eqnarray*}
and
\begin{eqnarray*}
U(x)=\frac{e^{C_m}}{c-1}\left(\log_{(m)}(x+e_{(m)})\right)^{1-c},
\end{eqnarray*}
which implies from \eqref{G.upper.at.infty.ap} that, for $c^{**}<\gamma+1$,
\begin{eqnarray*}
G^{**}_{h(x),x}(\infty) &\le& \frac{2+o(1)}{c^{**}-1}h(x)x\log x\ldots\log_{(m)}x
\quad\mbox{as }x\to\infty,
\end{eqnarray*}
and from \eqref{G.upper.at.infty.ap*}, for $c^*>\gamma+1$,
\begin{eqnarray*}
G^*_{h(x),x}(\infty) &\ge& \frac{2+o(1)}{c^*-1}h(x)x\log x\ldots\log_{(m)}x
\quad\mbox{as }x\to\infty.
\end{eqnarray*}
Repeating the arguments used in the proof of Theorem~\ref{thm:regular}, 
we obtain the desired result.
\end{proof}

\begin{proof}[Proof of Theorem \ref{thm:weibull}]
As in the proof of Theorem \ref{thm:regular}, 
from the very beginning we may assume
that $|\xi(y)|\le s(y)$ for all $y$ which implies both 
\eqref{cond.tail.left} and \eqref{cond.tail.right}.
Without loss of generality we assume that $h(x)\le s(x)$.
Let us choose a function $t(x)\uparrow\infty$ of order $o(h(x))$ as $x\to\infty$.

Fix some $c>0$ and consider $r(x)=cv(x)$. 
Then, by  l'H\^{o}spital's rule,
$$
\frac{U(x)}{U'(x)} \sim \frac{1}{r(x)}.  
$$ 
Therefore, as follows from \eqref{G.upper.at.infty.ap}
\begin{eqnarray}\label{regular_G_weibull}
G^{**}_{h(x),x}(\infty) &\le& (2+o(1))\frac{h(x)}{r(x)}\quad\mbox{as }x\to\infty,  
\end{eqnarray}
and from \eqref{G.upper.at.infty.ap*}
\begin{eqnarray}\label{regular_G_weibull*}
G^*_{h(x),x}(\infty) &\ge& (2+o(1))\frac{h(x)}{r(x)}\quad\mbox{as }x\to\infty.  
\end{eqnarray}

Considering $c^{**}<2/b$ and $c^*>2/b$ and repeating the arguments 
used in the proof of Theorem~\ref{thm:regular}, we conclude the proof.
\end{proof}

\section{Proof of local renewal theorem 
for asymptotically homogeneous Markov chains}

In this section, our purpose is to provide an approach that allows us 
to reduce the proof of the asymptotic behaviour of the renewal measure
on intervals to that on sufficiently slowly growing intervals.

\begin{lemma}\label{key:lem1}
Assume that there exist functions $v(x)>0$ and $\widetilde t(x)\uparrow\infty$ 
such that, for any $t(x)\uparrow\infty$ satisfying $t(x)\le\widetilde t(x)$,
\begin{eqnarray*}
\sup_{x\ge 1}\frac{v(x)H(x,x+t(x)]}{t(x)} &<& \infty.
\end{eqnarray*}
Then,
\begin{eqnarray}\label{finite.bound.U}
\sup_{x\ge 1} v(x) H(x,x+1] &<& \infty.
\end{eqnarray}
\end{lemma}	

\begin{proof}
Suppose that \eqref{finite.bound.U} fails. 
Then there exists a sequence $x_n\uparrow\infty$ such that 
$$
\alpha_n:= v(x_n) H(x_n,x_n+1] \to \infty\quad\mbox{as }n\to\infty. 
$$
Since both $\alpha_n$ and $\widetilde t(x_n)$ tend to infinity, there
exists a sequence  $t_n\uparrow\infty$ such that  
$t_n\le\widetilde t(x_n)$ and $t_n=o(\alpha_n)$ as $n\to\infty$. 
Let $t(x)$ be defined as follows 
$$
t(x) = t_n,\quad  x_n\le x< x_{n+1}. 
$$
Clearly, $t(x)\le\widetilde t(x)$ and $t(x)\uparrow\infty$. 
Then, eventually in $n$,
$$
\frac{v(x_n)H(x_n,x_n+t(x_n)]}{t(x_n)}
\ge 
\frac{v(x_n)H(x_n,x_n+1]}{t(x_n)} 
=\frac{\alpha_n}{t(x_n)}
\to\infty,
$$
which contradicts the hypothesis. 
\end{proof}	

\begin{proof}[Proof of Theorem \ref{thm:ah.renewal}.]
By Lemma~\ref{key:lem1} it follows from the assumption \eqref{eq.growing.intervals}   
that  the supremum in  (\ref{finite.bound.U}) is finite. 
In turn, it allows us to apply Helly's Selection Theorem 
to the family of measures $\{v(x)H(x+\cdot),\ x\in\R\}$ 
(see, for example, Theorem 2 in \cite[Section VIII.6]{Feller}). 
Hence, there exists a sequence of points $x_n\to\infty$ such that 
the sequence of measures $v(x_n)H(x_n+\cdot)$ converges weakly to some measure 
$\lambda$ as $n\to\infty$ in the standard sense of weak convergence
on bounded intervals. The following two results characterise $\lambda$.

\begin{lemma}\label{l.1}
Let $F$ denote the distribution of $\xi$.
A weak limit $\lambda$ of the sequence of measures $v(x_n)H(x_n+\cdot)$ 
satisfies the identity $\lambda=\lambda*F$.
\end{lemma}

\begin{proof}
The measure $\lambda$ is positive and $\sigma$-finite with necessity. 
Fix any smooth function $f(x)$ with a bounded support; 
let $A>0$ be such that $f(x)=0$ for $x\notin[-A,A]$. 
The weak convergence of measures means convergence of integrals
\begin{eqnarray}\label{conv.f.1}
\int_{-\infty}^\infty f(x)v(x_n)H(x_n+dx)
=\int_{-A}^A f(x)v(x_n)H(x_n+dx) \to \int_{-A}^A f(x)\lambda(dx)
\end{eqnarray}
as $n\to\infty$. On the other hand, due to the equality
$H(\cdot)=\P\{X_0\in\cdot\}+H * P(\cdot)$ we have the following
representation for the left side of \eqref{conv.f.1}:
\begin{equation}\label{conv.f.2}
\int_{-A}^A f(x)v(x_n)\P\{X_0\in x_n+dx\}
+\int_{-A}^A f(x)
\int_{-\infty}^\infty  P(x_n+y,x_n+dx)v(x_n)H(x_n+dy).
\end{equation}
Since $f$ and $v$ are  bounded,
\begin{equation}\label{conv.f.3}
\int_{-A}^A f(x)v(x_n)\P\{X_0\in x_n+dx\}
\le \|f\|_\infty \|v\|_\infty \P\{X_0\in[x_n-A,x_n+A]\} \to 0
\end{equation}
as $n\to\infty$. The second term in \eqref{conv.f.2} is equal to
\begin{eqnarray}\label{conv.f.4}
\int_{-\infty}^\infty 
v(x_n)H(x_n+dy)
\int_{-A}^A f(x)P(x_n+y,x_n+dx).
\end{eqnarray}
The weak convergence $P(t,t+\cdot)\Rightarrow F(\cdot)$ as $t\to\infty$ 
implies convergence of the inner integral in \eqref{conv.f.4}:
\begin{eqnarray*}
\int_{-A}^A f(x)P(x_n+y,x_n+dx)
&\to& \int_{-A}^A f(x)F(dx-y);
\end{eqnarray*}
here the rate of convergence can be estimated in the following way:
\begin{eqnarray*}
\Delta(n,y) &:=& \Biggl|\int_{-A}^A
f(x) (P(x_n+y,x_n+dx)-F(dx-y))\Biggr|\\
&=& \Biggl|\int_{-A}^A
f'(x)(\P\{\xi(x_n+y)\le x-y\}-F(x-y))dx\Biggr|\\
&\le& \|f'\|_\infty \int_{-A-y}^{A-y}
|\P\{\xi(x_n+y)\le x\}-F(x)|dx.
\end{eqnarray*}
Thus, the asymptotic homogeneity of the chain
yields for every fixed $C>0$ the uniform convergence
\begin{eqnarray}\label{Delta.1}
\sup_{y\in[-C,C]}\Delta(n,y) &\to& 0\quad\mbox{as }n\to\infty.
\end{eqnarray}
In addition, by the majorisation condition \eqref{majoriz}, for all $x\in\R$,
\begin{eqnarray*}
|\P\{\xi(x_n+y)\le x\}-F(x)| &\le& 2\P\{\Xi>|x|\}.
\end{eqnarray*}
Hence, for all $y$,
\begin{eqnarray}\label{Delta.2}
\Delta(n,y) &\le& 2\|f'\|_\infty \int_{-A-y}^{A-y}
\P\{\Xi>|x|\}dx\nonumber\\
&\le& 4A\|f'\|_\infty \P\{\Xi>|y|-A\}.
\end{eqnarray}
We have an estimate
\begin{eqnarray*}
\Delta_n &:=&
\Biggl|\int_{-\infty}^\infty v(x_n) H(x_n+dy)
\Biggl(\int_{-A}^A f(x)P(x_n{+}y,x_n{+}dx)
-\int_{-A}^A f(x)F(dx{-}y)\Biggr)\Biggr|\\
&\le& \int_{-\infty}^\infty \Delta(n,y) v(x_n) H(x_n+dy).
\end{eqnarray*}
For any fixed $C>0$, \eqref{Delta.1} and \eqref{finite.bound.U} imply that
\begin{eqnarray*}
\int_{-C}^C \Delta(n,y) v(x_n)H(x_n+dy)
&\le& \sup_{y\in[-C,C]} \Delta(n,y)
\cdot\sup_n \bigl( v(x_n)H[x_n-C,x_n+C]\bigr)\\
&\to& 0 \quad\mbox{as }n\to\infty.
\end{eqnarray*}
The remaining part of the integral can be estimated by \eqref{Delta.2}:
\begin{eqnarray*}
\lefteqn{\limsup_{n\to\infty}\int_{|y|\ge C} \Delta(n,y) v(x_n)H(x_n+dy)}\\
&\le& 4A\|f'\|_\infty \limsup_{n\to\infty}
\int_{|y|\ge C} \P\{\Xi>|y|-A\} v(x_n)H(x_n+dy).
\end{eqnarray*}
Since $\Xi$ has finite mean, property
\eqref{finite.bound.U} of the renewal measure $H$
allows us to choose a sufficiently large $C$
in order to make the `$\limsup$' as small as we please.
Therefore, $\Delta_n \to 0$ as $n\to\infty$.
Hence, \eqref{conv.f.4} has the same limit
as the sequence of integrals
\begin{eqnarray*}
\int_{-\infty}^\infty v(x_n)H(x_n+dy)
\int_{-A}^A f(x)F(dx-y).
\end{eqnarray*}
Now the weak convergence to $\lambda$
implies that \eqref{conv.f.4} has the limit
\begin{eqnarray}\label{conv.f.5}
\int_{-\infty}^\infty \lambda(dy)
\int_{-\infty}^\infty f(x)F(dx-y)
&=& \int_{-\infty}^\infty f(x)
\int_{-\infty}^\infty F(dx-y) \lambda(dy)\nonumber\\
&=& \int_{-\infty}^\infty f(x) (F*\lambda)(dx).
\end{eqnarray}
By \eqref{conv.f.1}--\eqref{conv.f.3}
and \eqref{conv.f.5}, we conclude the identity
\begin{eqnarray*}
\int_{-\infty}^\infty f(x)\lambda(dx)
&=& \int_{-\infty}^\infty f(x) (F*\lambda)(dx).
\end{eqnarray*}
Since this identity holds for every smooth function $f$ 
with a bounded support, the measures $\lambda$ and $F*\lambda$ coincide.
The proof is complete.
\end{proof} 

Further we use the following statement which is due to 
Choquet and Deny~\cite{CD}.

\begin{proposition}\label{l.2}
Let $F$ be a distribution not concentrated at $0$. Let $\lambda$ be a non-negative measure
satisfying the equality $\lambda=\lambda*F$ and the property
$\sup\limits_{n\in\Z}\lambda[n,n+1]<\infty$.

If $F$ is non-lattice, then $\lambda$ is proportional to the Lebesgue measure.

If $F$ is lattice with minimal span $1$ and $\lambda(\R\setminus\Z)=0$, 
then $\lambda$ is proportional to the counting measure.
\end{proposition}

The concluding part of the proof of Theorem \ref{thm:ah.renewal} 
will be carried out for the non-lattice case. 
Choose any sequence of points $x_n\to\infty$ such that the measure 
$v(x_n)H(x_n+\cdot)$ converges weakly to some measure $\lambda$ as $n\to\infty$. 
It follows from Lemma \ref{l.1} and Proposition \ref{l.2} that then 
$\lambda(dx)=\alpha\cdot dx$ with some $\alpha$, i.e.,
\begin{eqnarray*}
v(x_n)H(x_n+dx) &\Rightarrow& \alpha\cdot dx\ \mbox{ as }n\to\infty.
\end{eqnarray*}
Then, for any $A>0$ and $k\in\{0,1,2,\ldots\} $, 
$$
v(x_n)H(x_n+kA, x_n+(k+1)A] \to \alpha A.
$$
Then, there exists a sufficiently slowly growing sequence 
$t_n\uparrow\infty$ such that 
$$
\frac{v(x_n)H(x_n, x_n+t_n]}{t_n} \to \alpha.
$$
It follows from the assumption \eqref{eq.growing.intervals} that $\alpha=C_H$.  

We complete the proof by routine contradiction argument. 
Suppose there exists a sequence $\{x_n\}$ such that 
\begin{equation}\label{eq_contr}
v(x_n) H(x_n,x_n+h]\ \not\to\ C_Hh\quad\mbox{as }n\to\infty. 
\end{equation}
However, by  Helly's Selection Theorem and  arguments above there exists a further subsequence ${x_{n_k}}$ for 
which 
$$
v(x_{n_k})H(x_{n_k}, x_{n_k}+h] \to C_Hh,
$$
which contradicts \eqref{eq_contr}.
\end{proof}

\section{Random walks conditioned to stay positive}\label{sec:rwcsp}

In this section we prove Example~\ref{thm:rwcsp} by showing that
under the conditions stated the random walk conditioned to stay positive
satisfies all the conditions of Corollary~\ref{cor:regular}.
We start with checking that there is a function $s(x)\to\infty$ 
of order $o(x)$ such that
\[
m_1^{[s(x)]}\sim\frac{\sigma^2}{x}\quad\text{and}\quad 
m_2^{[s(x)]}\to\sigma^2\quad\mbox{as }x\to\infty,
\]
and~\eqref{regular_left_tail} holds 
for some decreasing integrable at infinity function $p(x)$.

Indeed, it is immediate from~\eqref{def.cond} that, for all $x$
such that $x-s(x)>0$,
\begin{align*}
m_1^{[s(x)]}(x)&
:=\frac{1}{V(x)}\E\{V(x+\xi_1)\xi_1;\ |\xi_1|\le s(x)\}\\
&=\frac{1}{V(x)}\E\{(V(x+\xi_1)-V(x))\xi_1;\ |\xi_1|\le s(x)\}
+\E\{\xi_1;\ |\xi_1|>s(x)\}\\
&=\frac{1}{V(x)}\E\{(V(x+\xi_1)-V(x))\xi_1;\ |\xi_1|\le s(x)\}+o(1/x),
\end{align*}
by $\E\xi_1=0$ and the finiteness of $\E\xi_1^2$,
provided $s(x)/x$ tends to zero sufficiently slow.
Finiteness of the second moment also implies that
ladder heights have finite expectation, so by the local renewal theorem,
\begin{eqnarray}\label{srt}
V(x+y)-V(x) &\to& \frac{y}{\E\chi^-}\quad\mbox{as }x\to\infty,
\end{eqnarray}
in non-lattice case;
in lattice case both $x$ and $y$ are restricted to the lattice.
Hence $(V(x+\xi_1)-V(x))\xi_1$ converges a.s.\ to
$\xi_1^2/\E\chi^-$ as $x\to\infty$.
By \eqref{srt}, $\sup_x(V(x+1)-V(x))=:c<\infty$ which yields
\begin{eqnarray}\label{incr.V}
|V(x+y)-V(x)| &\le& c_V(|y|+1).
\end{eqnarray}
This allows us to apply the dominated convergence theorem and to infer that
\begin{eqnarray*}
\E\{(V(x+\xi_1)-V(x))\xi_1;\ |\xi_1|\le s(x)\} &\to&
\frac{\E\xi_1^2}{\E\chi^-}\ =\ \frac{\sigma^2}{\E\chi^-}
\quad\mbox{as }x\to\infty.
\end{eqnarray*}
By the renewal theorem, $V(x)\sim x/\E\chi^-$ and hence
\begin{eqnarray}\label{m1}
m_1^{[s(x)]}(x) &\sim& \frac{\sigma^2}{x}\quad\mbox{as }x\to\infty.
\end{eqnarray}
For the truncated second moment of jumps we have
\begin{align*}
m_2^{[s(x)]}(x)&
:=\frac{1}{V(x)}\E\{V(x+\xi_1)\xi_1^2;\ |\xi_1|\le s(x)\}\\
&=\frac{1}{V(x)}\E\{(V(x+\xi_1)-V(x))\xi_1^2;\ |\xi_1|\le s(x)\}
+\E\{\xi_1^2;\ |\xi_1|\le s(x)\}\\
&=\frac{1}{V(x)}\E\{(V(x+\xi_1)-V(x))\xi_1^2;\ |\xi_1|\le s(x)\}+\sigma^2+o(1).
\end{align*}
Since for $|\xi_1|\le s(x),$
\begin{eqnarray*}
|V(x+\xi_1)-V(x)|\xi_1^2 &\le& c_V(1+|\xi_1|)\xi_1^2\ \le\ c_V(1+s(x))\xi_1^2
\end{eqnarray*}
so that 
\begin{eqnarray*}
\frac{|V(x+\xi_1)-V(x)|}{V(x)}\xi_1^2 &\stackrel{a.s.}\to& 0
\quad\mbox{as }x\to\infty,
\end{eqnarray*}
we get, again by the dominated convergence theorem,
\begin{eqnarray*}
\frac{1}{V(x)}\E\{(V(x+\xi_1)-V(x))\xi_1^2;\ |\xi_1|\le s(x)\}
&\to& 0\quad\text{as }x\to\infty.
\end{eqnarray*}
Therefore,
\begin{eqnarray*}
m_2^{[s(x)]}(x)\ \to\ \sigma^2\quad\mbox{as }x\to\infty.
\end{eqnarray*}
Summarizing,~\eqref{m1.m2.1x} holds with $\mu=\sigma^2$ and $b=\sigma^2$.
According to the construction of $X_n$,~\eqref{regular_left_tail} is equivalent to the following upper bound
\[
\frac{1}{V(x)}\E\{V(x+\xi_1);\ |\xi_1|>s(x)\}\ \le\ \frac{p(x)}{x}.
\]
Recalling that $V(x)$ is increasing and asymptotically linear, 
it suffices to show that
\[
\P\{\xi_1<-s(x)\}+\frac{1}{x}\E\{\xi_1;\ \xi_1>s(x)\}
\ \le\ \frac{p(x)}{x}
\]
for some $s(x)=o(x)$, but this is immediate from the assumption
$\E\xi_1^2<\infty$.

We also need to check the conditions 
\eqref{majorant_third}--\eqref{majorant_third_moment_exists}
and \eqref{majoriz}. To check the first one, we note that, 
\[
c_1\ :=\ \sup_x\frac{V(x+s(x))}{V(x)}\ <\ \infty,
\] 
hence, for $t\le s(x)=o(x)$,
\begin{eqnarray*}
\P\{|\xi(x)|>t, |\xi(x)|\le s(x)\}
&=& \biggl(\int_{-s(x)}^{-t}+\int_t^{s(x)}\biggr)\frac{V(x+u)}{V(x)}\P\{\xi_1\in du\}\\
&\le& c_1\P\{|\xi_1|>t\},
\end{eqnarray*}
and \eqref{majorant_third}--\eqref{majorant_third_moment_exists}
follows if we take $\widehat{\xi}$ defined by its tail as
\[
\P\{\widehat{\xi}>t\}=\min\{1,c_1\P\{|\xi_1|>t\}\},
\]
which is square integrable because $\xi_1$ is so.

Next, using once again \eqref{incr.V} we obtain
\begin{eqnarray*}
\P\{|\xi(x)|>t\}
&=& \biggl(\int_{-x}^{-t}+\int_t^\infty\biggr)
\frac{V(x+u)}{V(x)}\P\{\xi_1\in du\}\\
&\le& \P\{\xi_1<-t\}+\int_t^\infty
\Bigl(1+c_V\frac{u+1}{V(x)}\Bigr)\P\{\xi_1\in du\}\\
&\le& \P\{\xi_1<-t\}+\Bigl(1+\frac{c_V}{V(x)}\Bigr)\P\{\xi_1>t\}
+\frac{c_V}{V(x)}\E\{|\xi_1|;|\xi_1|>t\})\\
&\le& c_2(\P\{|\xi_1|>t\}+\E\{|\xi_1|;|\xi_1|>t\})\quad\mbox{for all }x,\ t>0.
\end{eqnarray*}
The right hand side is integrable due to $\E\xi_1^2<\infty$, 
so the condition \eqref{majoriz} is satisfied too.

Finally, the asymptotic homogeneity \eqref{asymp.hom} 
is immediate from \eqref{def.cond},  with $\xi=\xi_1$, 
because, for any fixed $u\in\R$, 
$V(x+u)/V(x)\to 1$ as $x\to\infty$, and the proof is complete.

\end{document}